\documentclass[10pt,english]{article}
\usepackage[utf8]{inputenc}
\usepackage{lmodern}

\usepackage[utf8]{inputenc}
\usepackage[a4paper]{geometry}
\usepackage{color}
\usepackage{xcolor}
\usepackage{mathrsfs}
\usepackage{mathtools}
\usepackage{bm}
\usepackage{amsmath}
\usepackage{amssymb}
\usepackage{amsthm}
\usepackage{hyperref}
\usepackage[english]{babel}
\usepackage{dsfont}
\usepackage{changepage}
\usepackage{kantlipsum}

\newcommand{\one}{\mathds{1}}

\def\d{\mathrm{d}}
\def\mod{\mathrm{Mod}}
\def\R{\mathbb{R}}
\def\N{\mathbb{N}}
\def\loc{\mathrm{loc}}

\hypersetup{colorlinks=true, linkcolor=blue, citecolor=blue}

\newtheorem{theorem}{Theorem}[section]

\newtheorem{lemma}[theorem]{Lemma}
\newtheorem{prop}[theorem]{Proposition}
\theoremstyle{definition}
\newtheorem{definition}[theorem]{Definition}

\numberwithin{equation}{section}

\def\Yint#1{\mathchoice
    {\YYint\displaystyle\textstyle{#1}}%
    {\YYint\textstyle\scriptstyle{#1}}%
    {\YYint\scriptstyle\scriptscriptstyle{#1}}%
    {\YYint\scriptscriptstyle\scriptscriptstyle{#1}}%
      \!\iint}
\def\YYint#1#2#3{{\setbox0=\hbox{$#1{#2#3}{\iint}$}
    \vcenter{\hbox{$#2#3$}}\kern-.51\wd0}}
\def\longdash{{-}\mkern-3.5mu{-}} 

\def\fiint{\Yint\longdash}

\DeclareMathOperator*{\esssup}{ess\,sup}
\DeclareMathOperator*{\essinf}{ess\,inf}
\DeclareMathOperator*{\essosc}{ess\,osc}
\DeclareMathOperator*{\supp}{supp}

\def\Xint#1{\mathchoice
	{\XXint\displaystyle\textstyle{#1}}%
	{\XXint\textstyle\scriptstyle{#1}}%
	{\XXint\scriptstyle\scriptscriptstyle{#1}}%
	{\XXint\scriptscriptstyle\scriptscriptstyle{#1}}%
	\!\int}
\def\XXint#1#2#3{{\setbox0=\hbox{$#1{#2#3}{\int}$ }
		\vcenter{\hbox{$#2#3$ }}\kern-.6\wd0}}

\def\dashint{\Xint-}

\makeatletter

\newcommand{\subjclass}[2][1991]{
  \let\@oldtitle\@title
  \gdef\@title{\@oldtitle\footnotetext{#1 \emph{Mathematics subject classification.} #2}}
}
\newcommand{\keywords}[1]{
  \let\@@oldtitle\@title
  \gdef\@title{\@@oldtitle\footnotetext{\emph{Key words and phrases.} #1.}}
}

\makeatother

\allowdisplaybreaks

\title{Variational solutions to the total variation flow on metric measure spaces}
\author{Vito Buffa\footnote{ORCID: 0000-0003-4175-4848. E-mail: bff.vti@gmail.com}, Juha Kinnunen\footnote{Aalto University, Department of Mathematics. E-mail: juha.k.kinnunen@aalto.fi}, Cintia Pacchiano Camacho\footnote{Aalto University, Department of Mathematics. E-mail: cintia.pacchiano@aalto.fi}}
\begin{document}

\maketitle

\begin{abstract}\noindent 
We discuss a purely variational approach to the total variation flow on metric measure spaces with a doubling measure and a Poincaré inequality.
We apply the concept of parabolic De Giorgi classes together with upper gradients, Newtonian spaces and functions of bounded variation to prove a necessary and sufficient condition for a variational solution to be continuous at a given point.  
\end{abstract}

\section{Introduction}
The total variation flow (TVF) is the partial differential equation
 \[
    \frac{\partial u}{\partial t} -\operatorname{div}\left(\frac{Du}{|Du|}\right)=0
    \quad\text{on}\quad \Omega_T=\Omega\times (0,T),
  \]
where $\Omega\subset\R^N$ is an open set and $T>0$.
There are several ways to define the concept of weak solution to the TVF.
One possibility is to apply the so-called Anzellotti pairing \cite{Anzellotti:1984}.
This approach has been applied in existence and uniqueness results for the total variation flow in \cite{AndreuBallesterCasellesMazon,AndreuCasellesDiazMazon,AndreuCasellesMazon,BellettiniCasellesNovaga}.
The variational inequality related to the TVF is
\[
     \int_0^T\|Du(t)\|(\Omega)\,\d t
     - \int_0^T\int_{\Omega}u(t)\frac{\partial\varphi}{\partial t}(t)\,\d x\,\d t
     \le\int_0^T\|D(u-\varphi)(t)\|(\Omega)\,\d t
   \]
 for every $\varphi\in C^\infty_0(\Omega_T)$, where the total variation $\|Du(t)\|(\Omega)$ is a Radon measure for almost every $t\in(0,T)$. 
A variational approach to existence and uniqueness questions has been
discussed by B\"ogelein, Duzaar and Marcellini \cite{BogeleinDuzaarMarcellini}, see also B\"ogelein, Duzaar and Scheven \cite{BogeleinDuzaarScheven}, and for the
corresponding obstacle problem in \cite{BoegelDuzSchev:2015}.
The natural function space for weak solutions to the TVF is bounded variation (BV).
Functions of bounded variation on metric measure spaces have been studied in \cite{AmbrosioDiMarino, mir}.
Instead of partial derivatives, this approach is based on the modulus of the gradient and using concepts such as minimal upper gradients and Newtonian spaces, 
see \cite{bjorn, hei, Heinonen1998, hkst, Koskela1998,Shanmugalingam}. This is also the main advantage of the variational approach.
A central motivation for developing such a theory has been the desire to unify the assumptions and methods employed in various specific spaces, such as Riemannian manifolds, Heisenberg groups, graphs, etc. 

The regularity theory of nonlinear parabolic problems in the metric space context has been developed and studied in \cite{Heran,Ivert_et_al,kmmp,MarolaMasson, Masson_et_al_2013,MassonParviainen,MassonSiljander}.
The comparison principle has been discussed in \cite{KinnunenMasson} and stability theory has been investigated in \cite{FujishimaHabermann, Fujishima_et_al,FujishimaHabermannKinnunenMasson}. 
Existence for parabolic problems on metric spaces has been discussed in \cite{collins}.
All of these results consider variational inequalities with $p$-growth for $p>1$. 
For the case $p=1$, which corresponds to the TVF, Buffa, Collins and Pacchiano \cite{BuffaCollinsPacchiano} showed existence of a parabolic minimizer using the concept of global variational solution. 
G\'orny and Maz\'on \cite{Gorny-Mazon:2021} studied the existence and uniqueness of weak solutions of the Neumann and Dirichlet problems to the TVF in metric measure spaces.
The main goal of the present paper is to extend the results of DiBenedetto, Gianazza and Klaus \cite{dbgk} to a metric measure space with a doubling measure and a Poincar\'e inequality.
The main result gives a necessary and sufficient condition for a variational solution to be continuous at a given point, see Theorem \ref{NecSuf}.
Our assumption on the time regularity of a variational solution is initially weaker than in \cite{dbgk} and thus our results may be interesting also in the Euclidean case.
As fas as we know, this is the first time when regularity questions are discussed for parabolic problems with linear growth on metric measure spaces.

The first step is to derive an energy estimate for variational solutions, in other words, to prove that variational solutions belong to a parabolic De Giorgi class, see Proposition \ref{belonging}.
The regularity results are based only on this energy estimate and on the assumptions made on the underlying metric measure space, but there is a technical difficulty present when establishing energy estimates for variational solutions.
It is not clear that the time regularity of a variational solution is a priori sufficient for placing it as the test function and performing the usual techniques used for obtaining an energy estimate. 
We resolve this issue by using a mollification technique. The idea of this technique is to prove the required energy estimate for mollified functions and finally to conclude the estimate at the limit. 
To establish the limiting estimate, we consider the upper gradient of a difference of functions, see Lemma \ref{vbthm}.
In the Euclidean case this poses no difficulties as we can use the linearity of the gradients, in the general metric setting the situation is not that simple, as taking an upper gradient is not a linear operation.

This paper is organized as follows. In Section \ref{Preliminaries} we recall basic definitions and describe the general setup of our study. 
Several results related to the function spaces and Sobolev-Poincar\'e inequalities may be of independent interest.
In Section \ref{TotalVariationFlow} we concentrate on the definition and properties of a variational solution to the TVF. 
Section \ref{ParabolicDeGiorgiClass} explores the relationship between variational solutions to the TVF and the parabolic De Giorgi classes. 
In Section \ref{DeGiorgiLemma} and \ref{TimeExpansion}, respectively, we prove that functions in a parabolic De Giorgi class are locally bounded and give a time expansion of positivity result. Finally, in Section \ref{CharacterizationContinuity} we present the characterization of continuity, i.e. we prove necessary and sufficient conditions for a variational solution to the TVF to be continuous at a given point. 
The last three sections are extensions of the corresponding results on Euclidean spaces by DiBenedetto, Gianazza and Klaus in \cite{dbgk} to metric measure spaces.

\section{Preliminaries}\label{Preliminaries}
\subsection{Newtonian spaces}

Let $(X,d, \mu)$ be a complete metric measure space endowed with a Borel measure $\mu$.
The measure $\mu$ is said to satisfy the doubling condition if there exists a constant $C_\mu\geq 1$,
called the doubling constant of $\mu$, such that
\begin{equation}\label{doubling}
0<\mu\left(B_{2r}(x)\right)\leq C_\mu\mu\left(B_{r}(x)\right)<\infty,
\end{equation}
for every $x\in X$ and $r>0$. 
Here $B_{r}(x)=\lbrace y\in X: d(x,y)<r\rbrace$ is an open ball centered at $x\in X$ with radius $r>0$.
We assume throughout that the measure $\mu$ is nontrivial in the sense that $0<\mu\left(B_{r}(x)\right)<\infty$ for every $x\in X$ and $r>0$.
A complete metric metric measure space with a doubling measure is proper, that is, closed and bounded subsets are compact, see \cite[Proposition 3.1]{bjorn}.
The doubling condition implies that for any $x\in X$, we have
\begin{equation}\label{desigualdadradios}
\frac{\mu(B_{R}(x))}{\mu(B_{r}(x))}\leq C\left(\frac{R}{r}\right)^{Q},
\end{equation}
for $0<r<R$ with $Q=\text{log}_2 \,C_\mu$ and $C=C_\mu^{-2}$, see \cite[Lemma 3.3]{bjorn}.
The exponent $Q=\text{log}_2 \,C_\mu$ is sometimes called the homogeneous dimension of $(X,d, \mu)$. 

A path $\gamma$ is a continuous mapping from a compact subinterval of $\mathbb R$ to $X$.
The $p$-modulus, with $1\le p<\infty$, of a path family $\Gamma$ on $X$ is
\[
\mod_p(\Gamma)
=\inf\int_X \rho^p\,\d\mu,
\]
where the infimum is taken over all nonnegative Borel functions $\rho$ with $\int_\gamma \rho\d s\ge1$ for all $\gamma\in\Gamma$, see \cite[Section 1.5]{bjorn}.
We recall the definition of upper gradient introduced and studied by \cite{Heinonen1998}, \cite{Koskela1998} and \cite{Shanmugalingam}.
General references for this theory are \cite{bjorn}, \cite{hei} and \cite{hkst}.

\begin{definition}
A nonnegative Borel function $g$ on $X$ is an upper gradient of a function $u:X\to[-\infty,\infty]$ if for all paths $\gamma$ in $X$, we have 
\begin{equation}\label{inequppergradient}
\vert u(x)-u(y)\vert\leq\int_{\gamma}g\, \d s,
\end{equation}
whenever both $u(x)$ and $u(y)$ are finite, and $\int_{\gamma}g\, \d s=\infty$ otherwise. Here $x$ and $y$ are the endpoints of $\gamma$.
Moreover, if a nonnegative $\mu$-measurable function $g$ satisfies \eqref{inequppergradient} for $p$-almost every path, 
that is with the exception of a path family of zero $p$-modulus, then $g$ is called a $p$-weak upper gradient of $u$.
\end{definition}

For $1\leq p<\infty$ and an open set $\Omega\subset X$, let
\[
\Vert u\Vert_{N^{1,p}(\Omega)}=\Vert u\Vert_{L^{p}(\Omega)}+\inf\Vert g\Vert_{L^{p}(\Omega)},
\]
where the infimum is taken over all upper gradients $g$ of $u$.
Consider the collection of all functions $u\in L^p(\Omega)$ with an upper gradient $g\in L^p(\Omega)$ and let
\begin{equation*}
\widetilde{N}^{1,p}(\Omega)
=\lbrace u:\Vert u\Vert_{N^{1,p}(\Omega)}<\infty\rbrace.
\end{equation*}
The Newtonian space is defined by
\begin{equation*}
N^{1,p}(\Omega)=\lbrace u:\Vert u\Vert_{N^{1,p}(\Omega)}<\infty\rbrace/\sim,
\end{equation*}
where $u\sim v$ if and only if $\Vert u-v\Vert_{N^{1,p}(\Omega)}=0$.

The corresponding local Newtonian space is defined by $u\in N^{1,p}_{\loc}(\Omega)$ if
$u\in N^{1,p}(\Omega')$ for all $\Omega'\Subset \Omega$, see \cite[Proposition 2.29]{bjorn}.
Here $\Omega'\Subset \Omega$ means that $\overline{\Omega'}$ is a compact subset of $\Omega$.
If $u$ has an upper gradient $g\in L^p(\Omega)$, there exists a unique minimal $p$-weak upper gradient $g_u\in L^p(\Omega)$ with
$g_u\le g$ $\mu$-almost everywhere for all $p$-weak upper gradients $g\in L^p(\Omega)$ of $u$, see  \cite[Theorem 2.5]{bjorn}.
Moreover, the minimal $p$-weak upper gradient is unique up to sets of measure zero.
For $u\in N^{1,p}(\Omega)$ we have
\[
\Vert u\Vert_{N^{1,p}(\Omega)}=\Vert u\Vert_{L^{p}(\Omega)}+\Vert g_u\Vert_{L^{p}(\Omega)},
\]
where $g_u$ is the minimal $p$-weak upper gradient of $u$.
The main advantage is that $p$-weak upper gradients behave better under $L^p$-convergence than upper gradients, see \cite[Proposition 2.2]{bjorn}.
However, the difference is relatively small, since every $p$-weak upper gradient can be approximated be a sequence of upper gradients in $L^p$, see \cite[Lemma 1.46]{bjorn}.
This implies that that the $N^{1,p}$-norm above remains the same if the infimum is taken over upper gradients instead of $p$-weak upper gradients.

We collect some calculus rules for upper gradients on metric measure spaces. 
Let $u,v\in N^{1,p}_{\mathrm{loc}}(\Omega)$ and let $g_{u},\,g_{v}\in L^{p}_{\mathrm{loc}}(\Omega)$ be the $p$-weak upper gradients of $u$ and $v$, respectively. 
Then $g_{u}+g_{v}$ and $\vert u\vert g_{v}+\vert v\vert g_{u}$ are $p$-weak upper gradients for $u+v$ and $uv$, respectively, see \cite[Theorem 2.15]{bjorn}.
Let  $\eta$ be Lipschitz continuous on $\Omega$ with $0\leq\eta\leq 1$ and consider $w=u+\eta(u-v)=(1-\eta)u+\eta v$. Then
$(1-\eta)g_{u}+\eta g_{v}+\vert v-u\vert g_{\eta}$
is a $p$-weak upper gradient of $w$, see \cite[Theorem 2.18]{bjorn}.
Moreover, $g_{u}=g_{v}$, $\mu$-almost everywhere on the set $\lbrace x\in X: u(x)=v(x)\rbrace$.
In particular, if $c\in\mathbb{R}$ is a constant, then $g_{u}=0$ $\mu-$almost everywhere on the set $\lbrace x\in X: u(x)=c\rbrace$, see \cite[Corollary 2.21]{bjorn}.

A metric measure space $(X,d,\mu)$ supports a weak Poincar\'e inequality, if there exist a constant $C_{P}$ and a dilation factor $\tau\geq 1$ 
such that for every ball $B_{\rho}(x_{0})$ in $X$, for every $u\in L^{1}_{\loc}(X)$ and every upper gradient $g$ of $u$, we have
\begin{equation}\label{qp-Poincare}
\dashint_{B_{\rho}(x_{0})}\vert u-u_{B_{\rho}(x_{0})}\vert\,\d\mu
\leq C_{P}\rho\dashint_{B_{\tau\rho}(x_{0})}g\,\d\mu,
\end{equation}
where the integral average is denoted by
\begin{equation*}
u_{B_{\rho}(x_{0})}=\dashint_{B_{\rho}(x_{0})} u\, \d\mu
=\frac{1}{\mu(B_{\rho}(x_{0}))}\int_{B_\rho(x_0)}u\,\d\mu.
\end{equation*}
A space supporting a Poincar\'e inequality is connected, see \cite[Proposition 4.2]{bjorn}.
Throughout the work, we assume that the measure $\mu$ is doubling and that the metric measure space $(X, d,\mu)$ supports a weak  Poincar\'e inequality. 
The weak Poincar\'e inequality and the doubling condition imply the Sobolev-Poincar\'e inequality
\begin{equation}\label{Sobolev-Poincare}
 \left(\dashint_{B_{\rho}(x_{0})} |u-u_{B_{\rho}(x_{0})}|^{\frac{Q}{Q-1}} \, \d \mu\right)^{\frac{Q-1}{Q}} 
 \le C\rho\dashint_{B_{2\tau\rho}(x_{0})}g\,\d\mu,
\end{equation}
for every $u\in L^1_{\loc}(X)$ and every $1$-weak upper gradient $g$ of $u$ and for every ball $B_{\rho}(x_{0})$ in $X$ with $C=C(C_\mu,C_P)$
and $Q$ as in \eqref{desigualdadradios}, see \cite[Theorem 4.21]{bjorn}. 

Next we discuss parabolic Newtonian spaces.

\begin{definition}\label{def.pnewtonian}
Let $\Omega\subset X$ be an open set, $0<T<\infty$ and $1\le p<\infty$. 
The parabolic Newtonian space $L^{p}(0, T;N^{1,p}(\Omega))$ consists of  strongly measurable functions 
$u:(0,T)\to N^{1,p}(\Omega)$ with the norm
\begin{equation*}
\Vert u\Vert_{L^{p}(0, T;N^{1,p}(\Omega))}
=\left(\int_{0}^{T}\Vert u(t)\Vert_{N^{1,p}(\Omega)}^{p}\,\d t\right)^{\frac1p}<\infty.
\end{equation*}
The integration over $(0,T)$ is taken with respect to the one-dimensional Lebesgue measure $\mathcal{L}^{1}$.
We say that $u\in L^{p}_{\loc}(0, T;N^{1,p}_{\loc}(\Omega))$ if for every 
$\Omega'\times (t_{1},t_{2})\Subset \Omega_{T}$ we have $u\in  L^{p}(t_{1}, t_{2};N^{1,p}(\Omega'))$.
Moreover, we denote $u\in L^{p}_{\textrm{c}}(0, T;N^{1,p}(\Omega))$ if for some $0<t_{1}<t_{2}<T$ we have $u(t)=0$ outside $[t_{1},t_{2}]$.
\end{definition}

The strong measurability of $u:(0,T)\to N^{1,p}(\Omega)$ and the assumption $u\in L^{p}(0, T;N^{1,p}(\Omega))$,
imply that there exists a sequence $(u_k)_{k\in\mathbb N}$ of simple functions
$u_k:(0,T)\to N^{1,p}(\Omega)$, 
\begin{equation}\label{simple-np}
u_k(t)=\sum_{i=1}^{n_k}\one_{E_i^{(k)}}(t)\cdot u_i^{(k)},
\end{equation}
where $\{E_i^{(k)}\}_{i=1}^{n_k}$ is a $\mathcal{L}^{1}$-measurable pairwise disjoint partition of $(0,T)$ and $v_i^{(k)}\in N^{1,p}(\Omega)$, $i=1,\dots,n_k$,
such that $u_k\to u$ in $L^{p}(0, T;N^{1,p}(\Omega))$ as $k\to\infty$.
In particular, we have $u_k(t)\to u(t)$ in $N^{1,p}(\Omega)$ for $\mathcal{L}^{1}$-almost every $t\in(0,T)$.
In other words, up to relabeling, we have
\begin{equation}\label{ae-lim-np}
 u(t)=\sum_{k=1}^\infty\one_{E_k}(t)\cdot u_k,
\end{equation}
with the sets $E_k$ and and simple functions $u_k$ as in \eqref{simple-np}. 

Next we consider upper gradients.
Since $u_k(t)\to u(t)$ in $N^{1,p}(\Omega)$ for $\mathcal{L}^{1}$-almost every $t\in(0,T)$ as $k\to\infty$, we have
$u(t)\in N^{1,p}(\Omega)$ for $\mathcal{L}^{1}$-almost every $t\in(0,T)$.
Consider the minimal $p$-weak upper gradient $g_{u(t)}\in L^p(\Omega)$ of $u(t)$ for $\mathcal{L}^{1}$-almost every $t\in(0,T)$.
The parabolic $p$-weak upper gradient of $u\in L^{p}(0, T;N^{1,p}(\Omega))$ is defined to be $g_u=g_{u(t)}$ for $\mathcal{L}^{1}$-almost every $t\in(0,T)$.

We note that the function $g_{u}$ is strongly measurable. 
For $\mathcal{L}^{1}$-almost every $t\in(0,T)$, the function $u(t)$ is the limit of strongly measurable functions $u_k(t)$ defined in \eqref{simple-np}.
By \eqref{ae-lim-np} and the locality of minimal $p$-weak upper gradients, we have
\begin{equation}\label{strong-up-grad}
 g_{u(t)}
 = g_{\sum_{k=1}^\infty \one_{E_k}(t)\cdot u_k} 
 = \sum_{k=1}^\infty \one_{E_k}(t)\cdot g_{u_k},
\end{equation}
for $\mathcal{L}^{1}$-almost every $t\in(0,T)$. 
Strong measurability follows, since $u_k\in N^{1,p}(\Omega)$ and $g_{u_k}\in L^p(\Omega)$ for every $k\in\N$. 
In other words, $g_{u(t)}$ can be approximated in $L^p(\Omega)$ by the functions $g_{u_k}(t)\in L^p(\Omega)$,
\begin{equation*}
 g_{u_k(t)} = g_{\sum_{i=1}^{n_k} \one_{E_i^{(k)}}(t)\cdot u_i^{(k)}} = \sum_{i=1}^{n_k} \one_{E_i^{(k)}}(t)\cdot g_{u_i^{(k)}},
\end{equation*}
which we obtain from \eqref{simple-np} by arguing as in \eqref{strong-up-grad}.
Since $u_k\to u$ in $L^{p}(0, T;N^{1,p}(\Omega))$,
we have $u_k\to u$  in $L^{p}(0, T;L^{p}(\Omega))$ and
$g_{u_k}\to g_u$ in $L^{p}(0, T;L^{p}(\Omega))$ as $k\to\infty$.

The product measure in the space $X\times (0,T)$, $T>0$, is denoted by $\mu\otimes\mathcal{L}^{1}$.
For $T>0$, we denote the space-time cylinder over an open subset $\Omega\subset X$ as $\Omega_{T}=\Omega\times(0, T)$.
For $u\in L^{p}(0, T;L^{p}(\Omega))$, there exists a $(\mu\otimes\mathcal{L}^{1})$-measurable representative
$u:\Omega_T\to[-\infty,\infty]$ such that $u(t)=u(\cdot,t)$ for $\mathcal{L}^{1}$- almost every $t\in(0,T)$ and
\[
\int_0^T\int_\Omega|u(x,t)|^p\,\d\mu\,\d t
=\int_0^T\Vert u(t)\Vert_{L^p(\Omega)}^p\,\d t.
\]
See \cite[Theorem 23.21] {Kuttler} and \cite[Section 2.1.1]{ruz}).
Similarly, for $g_u\in L^{p}(0, T;L^{p}(\Omega))$, there exists a $(\mu\otimes\mathcal{L}^{1})$-measurable representative
$g_u:\Omega_T\to[-\infty,\infty]$ such that $g_u(t)=g_u(\cdot,t)$ for $\mathcal{L}^{1}$-almost every $t\in(0,T)$.

With these observations we may consider the parabolic Newtonian space $L^{p}(0, T;N^{1,p}(\Omega))$ to be the space of functions 
$u\in L^p(\Omega_T)$, with $u=u(x,t)$, such that  for $\mathcal{L}^{1}$-almost every $t\in(0,T)$ the function $u(\cdot, t)$ belongs to $N^{1,p}(\Omega)$ and 
there exists $g_u\in L^p(\Omega_T)$ such that for $\mathcal{L}^{1}$-almost every $t\in(0,T)$ the function $g_u(\cdot, t)$ is a minimal $p$-weak upper gradient of $u(\cdot, t)$
with
\begin{equation*}
\iint_{\Omega_T}\left(|u(x,t)|^p+|g_u(x,t)|^p\right)\,\d\mu\,\d t
<\infty.
\end{equation*}

Let $u\in L^p_{\mathrm{loc}}(0,T; N^{1,p}_{\mathrm{loc}}(\Omega))$,  $1\le p<\infty$, and consider the time mollification
\[
u_\varepsilon(t)
=\int_{-\varepsilon}^{\varepsilon}\eta_\varepsilon(s)u(t-s)\,\d s,
\]
where $\eta_\varepsilon(s)=\frac1s\eta(\frac s\varepsilon)$, $\varepsilon>0$, is a standard mollifier.
The following approximation result was proved in more generality in \cite{Buffa}. 
We include a slightly modified version together with its full proof for reader's convenience.
We say that $u_\varepsilon\to u$ in $L^{p}_{\loc}(0, T;N^{1,p}_{\loc}(\Omega))$, if 
$\|u_\varepsilon-u\|_{L^{p}(t_{1}, t_{2};N^{1,p}(\Omega'))}\to 0$ as $\varepsilon\to0$
for every $\Omega'\times (t_{1},t_{2})\Subset \Omega_{T}$, where $\Omega'\Subset\Omega$ and $0<t_1<t_2<T$.

\begin{lemma}\label{vbthm}
Let $\Omega\subset X$ be an open set and assume that $u\in L^p_{\mathrm{loc}}(0,T; N^{1,p}_{\mathrm{loc}}(\Omega))$, $1\le p<\infty$.
Then $u_\varepsilon\to u$ in  $L^p_{\mathrm{loc}}(0,T; N^{1,p}_{\mathrm{loc}}(\Omega))$ as $\varepsilon\to0$.
In particular, we have $g_{u_\varepsilon-u}\to0$ in $L^p_{\mathrm{loc}}(\Omega_T)$ as $\varepsilon\to0$.
Moreover, as $s\to0$, we have $g_{u(\cdot,t-s)-u(\cdot,t)}\to0$ in $L^p_{\mathrm{loc}}(\Omega_T)$ uniformly in $t$.
\end{lemma}

\begin{proof} Since
\[
u(t)=\sum_{k=1}^\infty \one_{E_k}(t)\cdot u_k,
\]
for $\mathcal{L}^1$-almost every $t\in(0,T)$, by the definition of the time mollification we have
\begin{align*}
u_\varepsilon(t)&=\int_{-\varepsilon}^\varepsilon \eta_\varepsilon(s)u(t-s)\,\d s
=\int_{-\varepsilon}^\varepsilon \eta_\varepsilon(s)\sum_{k=1}^\infty \one_{E_k}(t-s)\cdot u_k\,\d s \\
&=\sum_{k=1}^\infty \left(\int_{-\varepsilon}^\varepsilon \eta_\varepsilon(s)\one_{E_k}(t-s)\,\d s\right)\cdot u_k
 = \sum_{k=1}^\infty (\one_{E_k})_\varepsilon(t)\cdot u_k, 
\end{align*}
which implies
\[
u(t)-u_{\varepsilon}(t)
=\sum_{k=1}^\infty \left(\one_{E_k}(t)-(\one_{E_k})_\varepsilon(t)\right)\cdot u_k,
\]
for $\mathcal{L}^1$-almost every $t\in(0,T)$.
By a standard mollifier argument we conclude that $u_\varepsilon\to u$ in $L^p_{\mathrm{loc}}(\Omega_T)$ as $\varepsilon\to0$.

By properties of minimal $p$-weak upper gradients and standard mollifications, we obtain
\begin{align*}
g_{u-u_\varepsilon} 
&=g_{\sum_{k=1}^\infty (\one_{E_k}-(\one_{E_k})_\varepsilon)\cdot u_k} 
\le \sum_{k=1}^\infty g_{(\one_{E_k}-(\one_{E_k})_\varepsilon)\cdot u_k} \\ 
&=\sum_{k=1}^\infty \left\vert\one_{E_k}-(\one_{E_k})_\varepsilon\right\vert\cdot g_{u_k}
 \xrightarrow{\varepsilon\rightarrow 0}0,
\end{align*}
$\mathcal{L}^1$-almost everywhere on $(0,T)$.
It follows that $g_{u-u_\varepsilon}\to0$  $\mu\otimes\mathcal{L} ^1$-almost everywhere in $\Omega_T$ as $\varepsilon\to0$.
Again, a standard mollifier argument implies $g_{u-u_\varepsilon}\to0$ in $L^p_{\mathrm{loc}}(\Omega_T)$ as $\varepsilon\to0$.

It remains to prove that $g_{u(t-s)-u(t)}\to0$ in $L^p_{\mathrm{loc}}(\Omega_T)$ as $s\to0$, uniformly in $t$. 
As above, we have
\[
 g_{u(t-s)-u(t)}=\sum_{k=1}^\infty\left\vert\one_{E_k}(t-s)-\one_{E_k}(t)\right\vert\cdot g_{u_k},
\]
for every $s>0$.
 
Let $\Omega'\times (t_{1},t_{2})\Subset \Omega_{T}$. 
By Fubini's theorem and Minkowski's inequality, we have
\begin{align*}
&\left( \int_{\Omega'\times (t_1,t_2)} g_{u(t-s)-u(t)}^p\,\d\mu\,\d t\right)^\frac{1}{p}
=\left(\int_{\Omega'\times (t_1,t_2)} \left(\sum_{k=1}^\infty\left\vert\one_{E_k}(t-s)-\one_{E_k}(t)\right\vert\cdot g_{u_k}\right)^p\,\d\mu\,\d t\right)^\frac{1}{p} \\ 
&\qquad\le \left(\int_{t_1}^{t_2}\int_{\Omega'}\left(\sum_{k=1}^\infty\left\vert\one_{E_k}(t-s)-\one_{E_k}(t)\right\vert\right)^p
\left(\sum_{k=1}^\infty g_{u_k}\right)^p\,\d\mu\,\d t\right)^\frac{1}{p} 
\\
&\qquad=\left(\int_{t_1}^{t_2}\left(\sum_{k=1}^\infty\left\vert\one_{E_k}(t-s)-\one_{E_k}(t)\right\vert\right)^p\,\d t\right)^\frac{1}{p}
\left(\int_{\Omega'}\left(\sum_{k=1}^\infty g_{u_k}\right)^p\,\d\mu\right)^\frac{1}{p} \\
&\qquad=\left\Vert\sum_{k=1}^\infty\left(\one_{E_k}(t-s)-\one_{E_k}(t)\right)\right\Vert_{L^p( (t_{1},t_{2}))}
\left\Vert\sum_{k=1}^\infty g_{u_k}\right\Vert_{L^p(\Omega')} \\ 
&\qquad\le \sum_{k=1}^\infty \|\one_{E_k}(t-s)-\one_{E_k}(t)\|_{L^p( (t_{1},t_{2}))}\cdot\|g_{u_k}\|_{L^p(\Omega')} \\ 
&\qquad=\sum_{k=1}^\infty\left(\int_{t_1}^{t_2}\left\vert\one_{E_k}(t-s)-\one_{E_k}(t)\right\vert^p\,\d t\right)^\frac{1}{p}\left(\int_{\Omega'} g_{u_k}^p\,\d\mu\right)^\frac{1}{p}.
\end{align*}
Since $\one_{E_k}\in L^p(0,T)$, $k\in\N$, the expression above vanishes as $s\to0$ by the continuity of translations on $L^p$ functions.
\end{proof}

\subsection{$BV$ functions}

Next we recall the definition and basic properties of functions of bounded variation on metric spaces, see \cite{mir}.
The total variation of $u\in L_{\mathrm{loc}}^{1}(X)$ is defined as
\begin{equation*}
\Vert Du\Vert(X)
=\inf\left\{\liminf_{i\rightarrow\infty}\int_X g_{u_{i}}\,\d\mu\right\},
\end{equation*}
where the infimum is taken over all sequences $(u_i)_{i\in\N}$ with $u_{i}\in \textrm{Lip}_{\loc}(X)$ for every $i\in\N$ and $u_{i}\rightarrow u\ \textrm{in}\ L^{1}_{\loc}(X)$ as $i\to\infty$.
Here $g_{u_{i}}$ is a 1-weak upper gradient of $u_{i}$ and $\textrm{Lip}_{\loc}(X)$ denotes the class of functions that are Lipschitz continuous on compact subsets of $X$. We say that a function $u\in L^{1}(X)$ is of bounded variation, and denote $u\in BV(X)$, 
if $\Vert Du\Vert(X)<\infty$. 
By replacing $X$ with an open set $\Omega\subset X$ in the definition of the total variation, we can define $\Vert Du\Vert(\Omega)$. 
A function $u\in BV_{\loc}(\Omega)$ if $u\in BV(\Omega')$ for all open sets $\Omega'\Subset \Omega$.
For an arbitrary set $A\subset X$, we set
\begin{equation*}
\Vert Du\Vert(A)=\inf\lbrace\Vert Du\Vert(U): A\subset U,\ U\subset X\ \textrm{is open}\rbrace.
\end{equation*}
If $u\in BV(\Omega)$, then $\Vert Du\Vert(A)$ is a finite Radon measure on $\Omega$ by \cite[Theorem 3.4]{mir}.
For the following result, see \cite[Theorem 4.3]{Lahti2017}.

\begin{theorem}
Let $\Omega\subset X$ be an open set and $u\in L_{\mathrm{loc}}^{1}(\Omega)$. If $\|Du\|(\Omega)<\infty$, then
\[
\Vert Du\Vert(\Omega)
=\inf\left\{\liminf_{i\rightarrow\infty}\int_\Omega g_{u_{i}}\,\d\mu: 
u_{i}\in N^{1,1}_{\loc}(\Omega),\ u_{i}\rightarrow u\ \textrm{in $L^{1}_{\loc}(\Omega)$ as $i\to\infty$}\right\},
\]
where $g_{u_i}$ is the minimal $1$-weak upper gradient of $u_i$ in $\Omega$.
\end{theorem}

If the space supports the Poincar\'e inequality in \eqref{qp-Poincare}, by an approximation argument, for
every $u\in BV_{\loc}(X)$ and every ball $B_\rho(x_{0})$ in $X$, we have
\begin{equation}
\label{Sobolev-PoincareBV}
\begin{split}
\dashint_{B_\rho(x_{0})} |u-u_{B_\rho(x_{0})}|\,\d\mu 
&\le\left(\dashint_{B_{\rho}(x_{0})} |u-u_{B_{\rho}(x_{0})}|^{\frac{Q}{Q-1}}\,\d\mu\right)^{\frac{Q-1}{Q}} \\
&\le C\rho\frac{\|Du\|(B_{2\tau\rho}(x_{0}))}{\mu(B_{2\tau\rho}(x_{0}))},
\end{split}
\end{equation}
where the constant $C$ and the dilation factor $\tau$ are the same as in the Sobolev-Poincar\'e inequality in \eqref{Sobolev-Poincare}.
Next we state a Sobolev type inequality for $BV$ functions which vanish on a large set, see \cite{KinnunenEtAl2014} and \cite[Theorem 5.51]{bjorn} for the corresponding result for Newtonian spaces.

\begin{theorem}\label{Sobolevzero}
Assume that $\mu$ is doubling and that $(X,d,\mu)$ supports a Poincar\'e inequality. 
Then there exists a constant $C=C(C_\mu,C_P)$ such that if $B_\rho(x_{0})$ is a ball in $X$ with $0<\rho<\frac{1}{4}\mathrm{diam}X$ 
and $u\in BV(X)$ with $u=0$ in $X\setminus B_\rho(x_{0})$, then
\begin{equation}\label{sobolevineq}
\left(\dashint_{B_\rho(x_{0})}\vert u \vert^{\frac Q{Q-1}}\,\d\mu\right)^{\frac{Q-1}Q}
\leq C\rho\frac{\|Du\|(B_{\rho}(x_{0}))}{\mu(B_{\rho}(x_{0}))},
\end{equation}
where $Q$ is as in \eqref{desigualdadradios}.
\end{theorem}

\begin {proof}
By Minkowski's inequality and \eqref{Sobolev-PoincareBV} we have
\[
\begin{split}
\left(\dashint_{B_{2\rho}(x_{0})}\vert u \vert^{\frac Q{Q-1}}\,\d\mu\right)^{\frac{Q-1}Q}
&\le\left(\dashint_{B_{2\rho}(x_{0})}\vert u-u_{B_{2\rho}(x_{0})} \vert^{\frac Q{Q-1}}\,\d\mu\right)^{\frac{Q-1}Q}
+|u_{B_{2\rho}(x_{0})}|
\\
&\le C\rho\frac{\|Du\|(B_{4\tau\rho}(x_{0}))}{\mu(B_{4\tau\rho}(x_{0}))}
+|u_{B_{2\rho}(x_{0})}|.
\end{split}
\]
H\"older's inequality and the fact that $u=0$ in $B_{2\rho}(x_{0})\setminus B_{\rho}(x_{0})$ imply that
\[
\begin{split}
|u_{B_{2\rho}(x_{0})}|
&\le\dashint_{B_{2\rho}(x_{0})}\vert u \vert\,\d\mu
=\dashint_{B_{2\rho}(x_{0})}\vert u \vert\chi_{B_{\rho}(x_{0})}\,\d\mu\\
&\le\left(\frac{\mu(B_{\rho}(x_{0}))}{\mu(B_{2\rho}(x_{0}))}\right)^{\frac1Q}
\left(\dashint_{B_{2\rho}(x_{0})}\vert u \vert^{\frac Q{Q-1}}\,\d\mu\right)^{\frac{Q-1}Q}.
\end{split}
\]
By \cite[Lemma 3.7]{bjorn} we have $\frac{\mu(B_{\rho}(x_{0}))}{\mu(B_{2\rho}(x_{0}))}\le\gamma<1$, where $\gamma=C(C_\mu)$, and we obtain
\[
(1-\gamma^{\frac1Q})\left(\dashint_{B_{2\rho}(x_{0})}\vert u \vert^{\frac Q{Q-1}}\,\d\mu\right)^{\frac{Q-1}Q}
\le C\rho\frac{\|Du\|(B_{4\tau\rho}(x_{0}))}{\mu(B_{4\tau\rho}(x_{0}))}
\le C\rho\frac{\|Du\|(B_{\rho}(x_{0}))}{\mu(B_{\rho}(x_{0}))},
\]
where we used the fact that $u=0$ in $B_{4\rho}(x_{0})\setminus B_{\rho}(x_{0})$ and thus $\|Du\|(B_{4\rho}(x_{0})\setminus B_{\rho}(x_{0}))=0$. The claim follows since $0<\gamma<1$.
\end{proof}

The following isoperimetric inequality in \cite[Lemma 2.2]{dbgv} has been originally obtained by De Giorgi.
See also \cite[Lemma 5.2]{cozzi} for the case $p>1$.
We give a proof that is based on the Sobolev-Poincar\'e type inequality for $BV$ in \eqref{Sobolev-PoincareBV}.
\begin{lemma}\label{DeGiorgiBV}
Assume that $\mu$ is doubling and that $(X,d,\mu)$ supports a Poincar\'e inequality.  
Then there exists a constant $C=C(C_\mu,C_P)$ such that if $B_{\rho}(x_{0})$ is a ball in $X$ and $u\in BV_{\loc}(X)$, then for $k<l$ real numbers we get
\begin{equation*}
\frac{(l-k)\mu(B_{\rho}(x_{0})\cap\{u>l\})}{\mu(B_{\rho}(x_{0}))}
\leq\frac{C\rho}{\mu(B_{\rho}(x_{0})\cap\{u\leq k\})}\Vert Du\Vert(\{k<u<l\}).
\end{equation*}
\end{lemma}

\begin{proof}
Let
\begin{equation*}
v=\begin{cases}
\min\lbrace u,l\rbrace-k, & \textrm{if}\ u>k,\\
0,&\textrm{if}\ u\leq k.
\end{cases} 
\end{equation*}
Notice that
\begin{align*}
(l-k)\mu(B_{\rho}(x_{0})\cap\{u>l\})^{\frac Q{Q-1}}
&=\left(\int_{B_{\rho}(x_{0})\cap\{u>l\}}\vert v\vert^{\frac Q{Q-1}}\,\d\mu\right)^{\frac{Q-1}Q}\\
&\le\left(\int_{B_{\rho}(x_{0})}\vert v\vert^{\frac Q{Q-1}}\,\d\mu\right)^{\frac{Q-1}Q}.
\end{align*}
By the Sobolev-Poincar\'e inequality for $BV$ in \eqref{Sobolev-PoincareBV}, we have
\[
\begin{split}
\left(\dashint_{B_\rho(x_{0})}\vert v \vert^{\frac Q{Q-1}}\,\d\mu\right)^{\frac{Q-1}Q}
&\le\left(\dashint_{B_\rho(x_{0})}\vert v-v_{B_\rho(x_{0})}\vert^{\frac Q{Q-1}}\,\d\mu\right)^{\frac{Q-1}Q}
+\left\vert v_{B_\rho(x_{0})}\right\vert\\
&\le C\rho\frac{\|Dv\|(B_{2\tau\rho}(x_{0}))}{\mu(B_{2\tau\rho}(x_{0}))}
+\dashint_{B_\rho(x_{0})}\vert v \vert\,\d\mu,
\end{split}
\]
where
\[
\begin{split}
\dashint_{B_\rho(x_{0})}\vert v \vert\,\d\mu
&=\frac1{\mu(B_\rho(x_{0}))}\int_{B_\rho(x_{0})\cap\{u>k\}}\vert v \vert\,\d\mu\\
&\le\frac1{\mu(B_\rho(x_{0}))}\left(\int_{B_\rho(x_{0})\cap\{u>k\}}\vert v \vert^{\frac Q{Q-1}}\,\d\mu\right)^{\frac{Q-1}Q}
\mu(B_\rho(x_{0})\cap\{u>k\})^{\frac1Q}\\
&=\left(\frac{\mu(B_\rho(x_{0})\cap\{u>k\})}{\mu(B_\rho(x_{0}))}\right)^{\frac1Q}
\left(\dashint_{B_\rho(x_{0})}\vert v \vert^{\frac Q{Q-1}}\,\d\mu\right)^{\frac{Q-1}Q}.
\end{split}
\]
This implies
\begin{equation}\label{SobolevIneq}
\left(1-\left(\frac{\mu(B_\rho(x_{0})\cap\{u>k\})}{\mu(B_\rho(x_{0}))}\right)^{\frac1Q}\right)
\left(\dashint_{B_\rho(x_{0})}\vert v \vert^{\frac Q{Q-1}}\,\d\mu\right)^{\frac{Q-1}Q}
\le C\rho\frac{\|Dv\|(B_{2\tau\rho}(x_{0}))}{\mu(B_{2\tau\rho}(x_{0}))}.
\end{equation}

On the other hand
\[
\begin{split}
\left(\dashint_{B_\rho(x_{0})}\vert v \vert^{\frac Q{Q-1}}\,\d\mu\right)^{\frac{Q-1}Q}
&\ge(l-k)\left(\frac{\mu(B_\rho(x_{0})\cap\{u>l\})}{\mu(B_\rho(x_{0}))}\right)^{1-\frac1Q}\\
&\ge\frac{(l-k)\mu(B_\rho(x_{0})\cap\{u>l\})}{\mu(B_\rho(x_{0}))}
\left(\frac{\mu(B_\rho(x_{0})\cap\{u>k\})}{\mu(B_\rho(x_{0}))}\right)^{-\frac1Q}.
\end{split}
\]
By \eqref{SobolevIneq}, we have
\[
\left(\left(\frac{\mu(B_\rho(x_{0})\cap\{u>k\})}{\mu(B_\rho(x_{0}))}\right)^{-\frac1Q}-1\right)
\frac{(l-k)\mu(B_\rho(x_{0})\cap\{u>l\})}{\mu(B_\rho(x_{0}))}
\le C\rho\frac{\|Dv\|(B_{2\tau\rho}(x_{0}))}{\mu(B_{2\tau\rho}(x_{0}))},
\]
where, by  Bernoulli's inequality, we have
\[
\begin{split}
\left(\frac{\mu(B_\rho(x_{0})\cap\{u>k\})}{\mu(B_\rho(x_{0}))}\right)^{-\frac1Q}-1
&=\left(\frac{\mu(B_\rho(x_{0}))}{\mu(B_\rho(x_{0})\cap\{u>k\})}\right)^{\frac1Q}-1\\
&\ge\frac1Q\left(\frac{\mu(B_\rho(x_{0}))}{\mu(B_\rho(x_{0})\cap\{u>k\})}-1\right)\\
&=\frac1Q\frac{\mu(B_\rho(x_{0})\cap\{u\le k\})}{\mu(B_\rho(x_{0})\cap\{u>k\})}\\
&\ge\frac1Q\frac{\mu(B_\rho(x_{0})\cap\{u\le k\})}{\mu(B_\rho(x_{0}))}.
\end{split}
\]
This implies
\[
\frac{(l-k)\mu(B_\rho(x_{0})\cap\{u>l\})}{\mu(B_\rho(x_{0}))}
\le\frac{CQ\rho}{\mu(B_{\rho}(x_{0})\cap\{u\leq k\})}
\|Dv\|(B_{2\tau\rho}(x_{0})),
\]
where
\begin{align*}
\|Dv\|(B_{2\tau\rho}(x_{0}))&=\|Dv\|(B_{2\tau\rho}(x_{0})\cap\{k<u<l\})+\|Dv\|(B_{2\tau\rho}(x_{0})\setminus\{k<u<l\})\\
&=\|Dv\|(B_{2\tau\rho}(x_{0})\cap\{k<u<l\})\\
&=\|D(u-k)\|(B_{2\tau\rho}(x_{0})\cap\{k<u<l\})\\
&=\|Du\|(B_{2\tau\rho}(x_{0})\cap\{k<u<l\})\\
&\leq\|Du\|(\{k<u<l\}).
\end{align*}
This proves the claim.
\end{proof}

We also apply parabolic BV functions.

\begin{definition}
Let $\Omega\subset X$ be an open set and $0<T<\infty$. 
We consider a parabolic $BV$ space $L^{1}(0, T;BV(\Omega))$, which consists of  functions 
$u:(0,T)\to BV(\Omega)$ such that 
\[
\int_0^T\left(\Vert u(t)\Vert_{L^1(\Omega)}+\Vert Du(t)\Vert (\Omega)\right)\,\d t<\infty.
\]
Here 
\[
\Vert Du(t)\Vert(\Omega)
=\inf\left\{\liminf_{i\rightarrow\infty}\int_\Omega g_{u_{i}}(t)\,\d\mu\right\},
\]
where the infimum is taken over all sequences $(u_i)_{i\in\N}$, with
$u_{i}\in L^1_{\loc}(0,T;N_{\loc}^{1,1}(\Omega))$ for every $i\in\N$ and $u_{i}\rightarrow u$ in $L^1_{\loc}(0,T;N_{\loc}^{1,1}(\Omega))$ as $i\to\infty$.
We say that $u\in L^{1}_{\loc}(0, T;BV_{\loc}(\Omega))$, if for every $\Omega'\times (t_{1},t_{2})\Subset \Omega_{T}$, we have $u\in  L^{1}(t_{1}, t_{2};BV(\Omega'))$.
\end{definition}

Note that we do not assume strong measurability in the sense of Bochner, which is too restrictive for the parabolic $BV$ theory.
The space $L^{1}(0, T;BV(\Omega))$ satisfies a weaker measurability condition, see \cite{BuffaCollinsPacchiano},
which implies that $t\mapsto \Vert Du(t)\Vert$ is a Lebesgue measurable function on $(0,T)$.
For $u\in L^{1}(0, T;BV(\Omega))$ there exists a $(\mu\otimes\mathcal{L}^{1})$-measurable function
$u:\Omega_T\to[-\infty,\infty]$ such that $u(\cdot,t)\in BV(\Omega)$ for $\mathcal{L}^{1}$-almost every $t\in(0,T)$.

Next we consider a Sobolev inequality for the parabolic $BV$.

\begin{prop}\label{1ap} 
There exists a constant $C=C(C_\mu,C_P)$, such that if $B_\rho(x_{0})$ is a ball in $X$ with $0<\rho<\frac{1}{4}\mathrm{diam}X$ 
and $u\in L_{\loc}^{1}(0,T;BV_{\loc}(X))$ with $u=0$ in $(X\setminus B_\rho(x_{0}))\times(0,T)$, then
\[
\int_{t_1}^{t_2}\int_{B_\rho(x_{0})}|u(t)|^{\kappa}\,\d\mu\,\d t
\le C\rho\int_{t_1}^{t_2}
\Vert Du(t)\Vert(B_\rho(x_{0}))\,\d t
\left(\underset{t_1<t<t_2}{\mathrm{esssup}}\,\dashint_{B_\rho(x_{0})}|u(t)|^{2}\,\d\mu\right)^{\frac1Q}.
\]
where $0<t_1<t_2<T$, $\kappa=\frac{Q+2}{Q}$ and $Q$ is as in \eqref{desigualdadradios}.
\end{prop}

\begin{proof}
H\"{o}lder's inequality and Sobolev's inequality \eqref{sobolevineq} imply
\begin{align*}
\int_{B_\rho(x_{0})}|u(t)|^{\kappa}\,\d\mu
&=\int_{B_\rho(x_{0})}|u(t)|^{1+(\kappa-1)}\,\d\mu\\
&\le\left(\int_{B_\rho(x_{0})}|u(t)|^{\frac{Q}{Q-1}}\d\mu\right)^{\frac{Q-1}Q}
\left(\int_{B_\rho(x_{0})}|u(t)|^{(\kappa-1)Q}\,\d\mu\right)^{\frac1Q}\\
 &=\left(\int_{B_\rho(x_{0})}|u(t)|^{\frac{Q}{Q-1}}\,\d\mu\right)^{\frac{Q-1}Q}
 \left(\int_{B_\rho(x_{0})}|u(t)|^{2}\,\d\mu\right)^{\frac1Q}\\
 &\le C\rho\|Du(t)\|(B_{\rho}(x_{0}))
 \left(\dashint_{B_\rho(x_{0})}|u(t)|^{2}\,\d\mu\right)^{\frac1Q},
\end{align*}
for $\mathcal{L}^{1}$-almost every $t\in(0,T)$.
The assertion follows by integrating over $(t_1,t_2)$.
\end{proof}

\section{Total variation flow}\label{TotalVariationFlow}

We discuss a definition of a variational solution to the total variation flow.

\begin{definition}\label{definitionsolution}
Let $\Omega\subset X$ be an open set and $0<T<\infty$. 
A function $u\in L^{1}_{\loc}(0,T; BV_{\loc}(\Omega))$ is a variational solution to the total variation flow in $\Omega_{T}$, if
\begin{equation}\label{parabolicminimizer}
\int_{0}^{T}\left(\int_{\Omega}-u(t)\frac{\partial\varphi}{\partial t}(t)\,\d\mu+\Vert Du(t)\Vert(\Omega)\right)\,\d t
\leq\int_{0}^{T}\Vert D(u+\varphi)(t)\Vert(\Omega)\,\d t,
\end{equation}
for every $\varphi\in\mathrm{Lip}(\Omega_T)$ with $\supp\varphi\Subset\Omega_T$.
\end{definition}

Boundary terms appear for test functions that do not necessarily vanish on the initial and the last moment of time.

\begin{prop}\label{lateralbdry} 
Let $u\in L_{\loc}^{1}(0,T; BV_{\loc}(\Omega))$ be a  variational solution to the total variation flow in $\Omega_T$ and let
$\Omega'\times (t_{1},t_{2})\Subset \Omega_{T}$, where $0<t_1<t_2<T$ are such that the boundary terms below are defined. 
Then
\begin{equation*}
\begin{split}
\int_{t_{1}}^{t_{2}}\left(\int_{\Omega'}-u(t)\frac{\partial\varphi}{\partial t}(t)\,\d\mu+\Vert Du(t)\Vert(\Omega')\right)\,\d t
&\leq\int_{t_{1}}^{t_{2}}\Vert D(u+\varphi)(t)\Vert(\Omega')\,\d t
-\left[\int_{\Omega'}u(t)\varphi(t)\,\d\mu\right]_{t=t_1}^{t_2},
\end{split}
\end{equation*}
for every $\varphi\in\mathrm{Lip}(\Omega_T)$ with $\supp\varphi\Subset \Omega'\times(0,T)$.
\end{prop}

\begin{proof}
Let $\varphi\in \textrm{Lip}(\Omega_{T})$ with $\textrm{supp }\varphi\Subset \Omega'\times (0,T)$. Let $\zeta_{h}$, $h>0$, be a cutoff function depending only on time, defined as
\begin{equation*}
\zeta_{h}(t)=
\begin{cases}0,&\quad 0\leq t<t_{1}-h,\\ 
\frac{1}{h}(t-t_{1}+h),&\quad t_{1}-h\leq t<t_{1},\\ 
1,&\quad t_{1}\leq t< t_{2},\\ 
-\frac{1}{h}(t-t_{2}-h),&\quad t_{2}\leq t<t_{2}+h,\\ 
0,&\quad t_{2}+h\leq t< T.
\end{cases} 
\end{equation*}
For small enough $h>0$, $\varphi_h=\varphi\zeta_{h}\in\textrm{Lip}(\Omega_T)$ with $\supp\varphi_{h}\Subset \Omega\times (0,T)$ and therefore admissible as a test function in the definition of variational solution. Thus 
 \begin{equation}\label{1}
 -\int_{0}^{T}\int_{\Omega}u(t)\frac{\partial\varphi_{h}}{\partial t}(t)\,\d\mu\,\d t+\int_{0}^{T}\Vert Du(t)\Vert(\Omega)\,\d t\leq \int_{0}^{T}\Vert D(u+\varphi_{h})(t)\Vert(\Omega)\,\d t.
 \end{equation}
Notice that 
\begin{equation*}
\frac{\partial\varphi_{h}}{\partial t}
=\frac{\partial\varphi}{\partial t}\zeta_{h}+\varphi\zeta_{h}'=\frac{\partial\varphi}{\partial t}\zeta_{h}+
\begin{cases}
0,&\quad 0\leq t<t_{1}-h,\\ 
\frac{1}{h}\left(\frac{\partial\varphi}{\partial t}(t-t_{1}+h)+\varphi\right),&\quad t_{1}-h\leq t<t_{1},\\ 
\frac{\partial\varphi}{\partial t},&\quad t_{1}\leq t< t_{2}, \\ 
-\frac{1}{h}\left(\frac{\partial\varphi}{\partial t}(t-t_{2}-h)+\varphi\right),&\quad t_{2}\leq t<t_{2}+h,\\ 
0,&\quad t_{2}+h\leq t< T.
\end{cases} 
\end{equation*}
For the first term on the left-hand side of \eqref{1}, we find
\begin{align*}
&-\int_{0}^{T}\int_{\Omega}u(t)\frac{\partial\varphi_{h}}{\partial t}(t)\,\d\mu\,\d t
\qquad=-\frac{1}{h}\int_{t_{1}-h}^{t_{1}}\int_{\Omega}u(t)\left(\frac{\partial\varphi}{\partial t}(t-t_{1}+h)+\varphi(t)\right)\,\d\mu\,\d t\\
 &\qquad\qquad-\int_{t_{1}}^{t_{2}}\int_{\Omega}u(t)\frac{\partial\varphi}{\partial t}(t)\,\d \mu\,\d t
 +\frac{1}{h}\int_{t_{2}}^{t_{2}+h}\int_{\Omega}u(t)\left(\frac{\partial\varphi}{\partial t}(t-t_{2}-h)+\varphi(t)\right)\,\d \mu\,\d t \\
 &\qquad=-\frac{1}{h}\int_{t_{1}-h}^{t_{1}}\int_{\Omega}u(t)\frac{\partial\varphi}{\partial t}(t-t_{1}+h)\,\d \mu\,\d t-\frac{1}{h}\int_{t_{1}-h}^{t_{1}}\int_{\Omega}u(t)\varphi(t)\,\d \mu\,\d t\\
  &\qquad\qquad-\int_{t_{1}}^{t_{2}}\int_{\Omega}u(t)\frac{\partial\varphi}{\partial t}(t)\,\d \mu\,\d t
  +\frac{1}{h}\int_{t_{2}}^{t_{2}+h}\int_{\Omega}u(t)\frac{\partial\varphi}{\partial t}(t-t_{2}-h)\,\d \mu\,\d t\\
  &\qquad\qquad+\frac{1}{h}\int_{t_{2}}^{t_{2}+h}\int_{\Omega}u(t)\varphi(t)\,\d \mu\,\d t.
\end{align*}
Since $\textrm{supp }\varphi\Subset \Omega'\times (0,T)$, we have
\begin{equation*}
-\int_{t_{1}}^{t_{2}}\int_{\Omega}u(t)\frac{\partial\varphi}{\partial t}(t)\,\d \mu\,\d t=-\int_{t_{1}}^{t_{2}}\int_{\Omega'}u(t)\frac{\partial\varphi}{\partial t}(t)\,\d \mu\,\d t.
\end{equation*}
By the dominated convergence theorem, we get
\[
-\frac{1}{h}\int_{t_{1}-h}^{t_{1}}\int_{\Omega}u(t)\frac{\partial\varphi}{\partial t}(t-t_{1}+h)\,\d \mu\,\d t
\xrightarrow{h\rightarrow0} 0,
\]
and
\[
\frac{1}{h}\int_{t_{2}}^{t_{2}+h}\int_{\Omega}u(t)\frac{\partial\varphi}{\partial t}(t-t_{2}-h)\,\d \mu\,\d t
\xrightarrow{h\rightarrow0} 0.
\]
on the other hand, by the Lebesgue differentiation theorem, we obtain
\[
-\frac{1}{h}\int_{t_{1}-h}^{t_{1}}\int_{\Omega}u(t)\varphi(t)\,\d \mu\,\d t\xrightarrow{h\rightarrow0} -\int_{\Omega'}u(t_{1})\varphi(t_{1})\,\d \mu
\]
and
\[
\frac{1}{h}\int_{t_{2}}^{t_{2}+h}\int_{\Omega}u(t)\varphi(t)\,\d \mu\,\d t\xrightarrow{h\rightarrow0} \int_{\Omega'}u(t_{2})\varphi(t_{2})\,\d \mu,
\]
for $\mathcal{L}^1$-almost every $t_1,t_2\in(0,T)$.
This implies that
\begin{equation}
-\int_{0}^{T}\int_{\Omega}u(t)\frac{\partial\varphi_{h}}{\partial t}(t)\,\d \mu\,\d t\xrightarrow{h\rightarrow0} -\int_{t_{1}}^{t_{2}}\int_{\Omega'}u(t)\frac{\partial\varphi}{\partial t}(t)\,\d \mu\,\d t+\left[\int_{\Omega'}u(t)\varphi(t)\,\d \mu\right]_{t=t_{1}}^{t_{2}},
\end{equation}
for $\mathcal{L}^1$-almost every $t_1,t_2\in(0,T)$.

For the second term on the left-hand side of \eqref{1}, we find
\begin{align*}
&\int_{0}^{T}\Vert Du(t)\Vert(\Omega)\,\d t
=\int_{0}^{t_{1}-h}\Vert Du(t)\Vert(\Omega)\,\d t+\int_{t_{1}-h}^{t_{1}}\Vert Du(t)\Vert(\Omega)\,\d t
+\int_{t_{1}}^{t_{2}}\Vert Du(t)\Vert(\Omega)\,\d t\\
&\qquad+\int_{t_{2}}^{t_{2}+h}\Vert Du(t)\Vert(\Omega)\,\d t+\int_{t_{2}+h}^{T}\Vert Du(t)\Vert(\Omega)\,\d t\\
&=\int_{0}^{t_{1}-h}\Vert Du(t)\Vert(\Omega)\,\d t+\int_{t_{1}-h}^{t_{1}}\Vert Du(t)\Vert(\Omega)\,\d t+\int_{t_{1}}^{t_{2}}\Vert Du(t)\Vert(\Omega\setminus \Omega')\,\d t\\
&\qquad+\int_{t_{1}}^{t_{2}}\Vert Du(t)\Vert(\Omega')\,\d t+\int_{t_{2}}^{t_{2}+h}\Vert Du(t)\Vert(\Omega)\,\d t+\int_{t_{2}+h}^{T}\Vert Du(t)\Vert(\Omega)\,\d t.
\end{align*}
Here
\[
\int_{t_{1}-h}^{t_{1}}\Vert Du(t)\Vert(\Omega)\,\d t\xrightarrow{h\rightarrow0}  0
\quad\text{and}\quad
\int_{t_{2}}^{t_{2}+h}\Vert Du(t)\Vert(\Omega)\,\d t\xrightarrow{h\rightarrow0}  0.
\]
Moreover,
\begin{align*}
&\int_{0}^{T}\Vert D(u+\varphi_{h})(t)\Vert(\Omega)\,\d t
=\int_{0}^{t_{1}-h}\Vert D(u+\varphi_{h})(t)\Vert(\Omega)\,\d t+\int_{t_{1}-h}^{t_{1}}\Vert D(u+\varphi_{h})(t)\Vert(\Omega)\,\d t\\
&\qquad\qquad+\int_{t_{1}}^{t_{2}}\Vert D(u+\varphi_{h})(t)\Vert(\Omega)\,\d t+\int_{t_{2}}^{t_{2}+h}\Vert D(u+\varphi_{h})(t)\Vert(\Omega)\,\d t\\
&\qquad\qquad+\int_{t_{2}+h}^{T}\Vert D(u+\varphi_{h})(t)\Vert(\Omega)\,\d t\\
&\qquad=\int_{0}^{t_{1}-h}\Vert Du(t)\Vert(\Omega)\,\d t
+\int_{t_{1}-h}^{t_{1}}\Vert D(u+\varphi_{h})(t)\Vert(\Omega)\,\d t
+\int_{t_{1}}^{t_{2}}\Vert D(u+\varphi_{h})(t)\Vert(\Omega)\,\d t\\
&\qquad\qquad+\int_{t_{2}}^{t_{2}+h}\Vert D(u+\varphi_{h})(t)\Vert(\Omega)\,\d t+\int_{t_{2}+h}^{T}\Vert Du(t)\Vert(\Omega)\,\d t,
\end{align*}
where
\[
\int_{t_{1}-h}^{t_{1}}\Vert D(u+\varphi_{h})(t)\Vert(\Omega)\,\d t\xrightarrow{h\rightarrow0}  0,
\quad
\int_{t_{2}}^{t_{2}+h}\Vert D(u+\varphi_{h})(t)\Vert(\Omega)\,\d t\xrightarrow{h\rightarrow0}  0,
\]
and
\begin{align*}
\int_{t_{1}}^{t_{2}}\Vert D(u+\varphi_{h})(t)\Vert(\Omega)\,\d t&\xrightarrow{h\rightarrow0}\int_{t_{1}}^{t_{2}}\Vert D(u+\varphi)(t)\Vert(\Omega)\,\d t \\
&=\int_{t_{1}}^{t_{2}}\Vert D(u+\varphi)(t)\Vert(\Omega\setminus\Omega')\,\d t +\int_{t_{1}}^{t_{2}}\Vert D(u+\varphi)(t)\Vert(\Omega')\,\d t \\
&=\int_{t_{1}}^{t_{2}}\Vert Du(t)\Vert(\Omega\setminus\Omega')\,\d t +\int_{t_{1}}^{t_{2}}\Vert D(u+\varphi)(t)\Vert(\Omega')\,\d t.
\end{align*}
Here we used the fact that $\textrm{supp }\varphi\Subset \Omega'\times (0,T)$.
Substituting these in \eqref{1}, we obtain
\begin{align*}
-&\int_{t_{1}}^{t_{2}}\int_{\Omega'}u(t)\frac{\partial\varphi}{\partial t}(t)\,\d \mu\,\d t
+\left[\int_{\Omega'}u(t)\varphi(t)\,\d \mu\right]_{t=t_{1}}^{t_{2}}+\int_{0}^{t_{1}-h}\Vert Du(t)\Vert(\Omega)\,\d t\\
&\qquad\qquad+\int_{t_{1}}^{t_{2}}\Vert Du(t)\Vert(\Omega')\,\d t+\int_{t_{2}+h}^{T}\Vert Du(t)\Vert(\Omega)\,\d t\\
&\qquad\leq \int_{0}^{t_{1}-h}\Vert Du(t)\Vert(\Omega)\,\d t +\int_{t_{1}}^{t_{2}}\Vert D(u+\varphi_{h})(t)\Vert(\Omega)\,\d t+\int_{t_{2}+h}^{T}\Vert Du(t)\Vert(\Omega)\,\d t.
\end{align*}
Eliminating the repeated elements gives
\begin{equation*}
\int_{t_{1}}^{t_{2}}\left(\int_{\Omega'}-u(t)\frac{\partial\varphi}{\partial t}(t)\,\d \mu+\Vert Du(t)\Vert(\Omega')\right)\,\d t\leq \int_{t_{1}}^{t_{2}}\Vert D(u+\varphi)(t)\Vert(\Omega')\,\d t-\left[\int_{\Omega'}u(t)\varphi(t)\,\d \mu\right]_{t=t_{1}}^{t_{2}}.
\end{equation*}

\end{proof}

\section{Parabolic De Giorgi class}\label{ParabolicDeGiorgiClass}
Next we define the class of functions for which we prove the regularity results.
For $(x_{0}, t_{0})\in X\times\mathbb{R}$ and $\rho,\theta>0$, we denote $Q^{-}_{\rho, \theta}(x_{0},t_{0})=B_{\rho}(x_{0})\times(t_{0}-\theta\rho,t_{0}]$.
The positive and negative parts of $u$ are denoted by $u_{\pm}=\max\lbrace\pm u,0\rbrace$, respectively.

\begin{definition}
A function $u\in L^{1}_{\loc}(0,T; BV_{\loc}(\Omega))$ belongs to the parabolic De Giorgi class $DG^{\pm}(\Omega_{T};\gamma)$, with $\gamma>0$, if
\begin{equation}\label{degiorgiclass}
\begin{split}
&\esssup_{t_{0}-\theta\rho\leq t\leq t_{0}}\int_{B_{\rho}(x_{0})}\varphi(t)(u(t)-k)_{\pm}^{2}\,\d \mu
+\int_{t_{0}-\theta\rho}^{t_{0}}\Vert D(\varphi(u-k)_{+})(t)\Vert(B_{\rho}(x_{0}))\,\d t\\\
&\qquad\leq\gamma\iint_{Q^{-}_{\rho, \theta}(x_{0},t_{0})}\left\vert\frac{\partial\varphi}{\partial t}(t)\right\vert(u(t)-k)_{\pm}^{2}\,\d \mu\,\d t
+\gamma\iint_{Q^{-}_{\rho, \theta}(x_{0},t_{0})}g_{\varphi}(t)(u(t)-k)_+\,\d \mu\,\d t\\
&\qquad\qquad-\left[\int_{B_{\rho}(x_{0})}\varphi(t)(u(t)-k)_{\pm}^2\,\d \mu\right]_{t=t_{0}-\theta\rho}^{t_0},
\end{split}
\end{equation}
for every $Q^{-}_{\rho, \theta}(x_{0},t_{0})\Subset \Omega_{T}$,  $k\in\mathbb{R}$ and $\varphi\in\textrm{Lip}(\Omega_T)$ 
with $\supp\varphi\Subset B_{\rho}(x_{0})\times(0,T)$ and $0\leq\varphi\leq 1$.

The parabolic De Giorgi class $DG(\Omega_{T};\gamma)$ is defined as 
\[
DG(\Omega_{T};\gamma)=DG^{+}(\Omega_{T};\gamma)\cap DG^{-}(\Omega_{T};\gamma).
\]
\end{definition}

The proof of the necessary and sufficient conditions for continuity of a variational solution to the total variation flow, Theorem \ref{NecSuf}, will only use the local integral inequalities in \eqref{degiorgiclass}.
We show that a variational solution to the total variation flow belongs to the parabolic De Giorgi class. 

\begin{prop}\label{belonging} Let $u$ be variational solution to the total variation flow in $\Omega_{T}$.
Then $u\in DG(\Omega_{T};8)$.
\end{prop}

\begin{proof}
Let $\phi\in\textrm{Lip}(\Omega_T)$ with $\supp\phi \Subset \Omega_{T}$.
There exists $h_{0}>0$ such that for every $0<h<h_{0}$, we have $\phi_{h}\in\textrm{Lip}(\Omega_T)$ with $\supp\phi_{h} \Subset \Omega_{T}$ and thus we may apply it as test function in \eqref{parabolicminimizer}. 
Here $\phi_{h}$ denotes the time mollification of $\phi$.
For a small enough $s$ the translated function $v(t)=u(t-s)$ fulfills \eqref{parabolicminimizer}. For $0<t_2<t_1<T$, to be specified later, Proposition \ref{lateralbdry} implies
\[
\begin{split}
&-\int_{t_2}^{t_1}\int_{B_{\rho}(x_{0})}v(t)\frac{\partial\phi_{h}}{\partial t}(t)\,\d \mu\,\d t
+\int_{t_2}^{t_1}\Vert Dv(t)\Vert(B_{\rho}(x_{0}))\,\d t\\
&\qquad\leq\int_{t_2}^{t_1}\Vert D(v+\phi_{h})(t)\Vert(B_{\rho}(x_{0}))\,\d t
-\left[\int_{B_{\rho}(x_0)}v(t)\phi_{h}(t)\,\d \mu\right]_{t=t_2}^{t_1}.
\end{split}
\]
Let $(u_i)_{i\in\N}$ be a minimizing sequence with $u_{i}\in L^1_{\loc}(0,T;N_{\loc}^{1,1}(\Omega))$ for every $i\in\N$, $u_{i}\rightarrow u$ in $L^1_{\loc}(0,T;N_{\loc}^{1,1}(\Omega))$ as $i\to\infty$
and
\[
\begin{split}
\int_{t_2}^{t_1}\Vert Dv(t)\Vert(B_{\rho}(x_{0}))\,\d t
&=\int_{t_2}^{t_1}\Vert Du(t-s)\Vert(B_{\rho}(x_{0}))\,\d t\\
&=\lim_{i\rightarrow\infty}\int_{t_2}^{t_1}\int_{B_{\rho}(x_{0})} g_{u_{i}}(t-s)\,\d\mu\,\d t\\
&=\lim_{i\rightarrow\infty}\int_{t_2}^{t_1}\int_{B_{\rho}(x_{0})} g_{v_{i}}(t)\,\d\mu\,\d t,
\end{split}
\]
where $v_i(t)=u_i(t-s)$ for every $i\in\N$.
Let $\epsilon>0$. There exists $i_\epsilon\in\N$ such that, for every $i\ge i_\epsilon$, we have
\[
\int_{t_2}^{t_1}\int_{B_{\rho}(x_{0})} g_{v_{i}}(t)\,\d\mu\,\d t
\le\int_{t_2}^{t_1}\Vert Dv(t)\Vert(B_{\rho}(x_{0}))\,\d t
+\frac{\epsilon}2,
\]
and 
\[
\int_{t_2}^{t_1}\Vert D(v+\phi_{h})(t))\Vert(B_{\rho}(x_{0}))\,\d t
\le\int_{t_2}^{t_1}\int_{B_{\rho}(x_{0})}g_{v_{i}+\phi_{h}}(t)\,\d\mu\,\d t
+\frac{\epsilon}2.
\]
This implies
\begin{align*}
-&\int_{t_2}^{t_1}\int_{B_{\rho}(x_{0})}v(t)\frac{\partial\phi_{h}}{\partial t}(t)\,\d \mu\,\d t
+\int_{t_2}^{t_1}\int_{B_{\rho}(x_{0})} g_{v_{i}}(t)\,\d\mu\,\d t\\
&\qquad\leq\int_{t_2}^{t_1}\int_{B_{\rho}(x_{0})}g_{v_{i}+\phi_{h}}(t)\,\d\mu\,\d t -\left[\int_{B_{\rho}(x_0)}v(t)\phi_{h}(t)\,\d \mu\right]_{t=t_2}^{t_1}
+\epsilon\\
&\qquad\leq\int_{t_2}^{t_1}\int_{B_{\rho}(x_{0})}g_{v_{i}+\phi}(t)\,\d\mu\,\d t
+\int_{t_2}^{t_1}\int_{B_{\rho}(x_{0})}g_{\phi_{h}-\phi}(t)\,\d\mu\,\d t
-\left[\int_{B_{\rho}(x_0)}v(t)\phi_{h}(t)\,\d \mu\right]_{t=t_2}^{t_1}+\epsilon,
\end{align*}
for every $i\ge i_\epsilon$.

Let $i\ge i_\epsilon$.
We multiply both sides of the inequality above by a standard mollifier $\eta_\varepsilon=\eta_\varepsilon(s)$ with support $[-\varepsilon,\varepsilon]$ for small enough $\varepsilon>0$. 
By integrating the resulting expression in the variable $s$ we obtain
\begin{align*}
-&\int_{-\varepsilon}^{\varepsilon}\int_{t_2}^{t_1}\int_{B_{\rho}(x_{0})}v(t)\eta_{\varepsilon}(s)\frac{\partial\phi_{h}}{\partial t}(t)\,\d \mu\,\d t\,\d s
+\int_{-\varepsilon}^{\varepsilon}\int_{t_2}^{t_1}\int_{B_{\rho}(x_{0})} g_{v_{i}}(t)\eta_{\varepsilon}(s)\,\d\mu\,\d t\,\d s\\
&\qquad\leq\int_{-\varepsilon}^{\varepsilon}\int_{t_2}^{t_1}\int_{B_{\rho}(x_{0})}g_{v_{i}+\phi}(t)\eta_{\varepsilon}(s)\,\d\mu\,\d t\,\d s
+\int_{-\varepsilon}^{\varepsilon}\int_{t_2}^{t_1}\int_{B_{\rho}(x_{0})}g_{\phi_{h}-\phi}(t)\eta_{\varepsilon}(s)\,\d\mu\,\d t\,\d s\\
&\qquad\qquad-\left[\int_{-\varepsilon}^{\varepsilon}\int_{B_{\rho}(x_0)}v(t)\eta_\varepsilon(s)\phi_{h}(t)\,\d \mu\,\d s\right]_{t=t_2}^{t_1}
+\epsilon.
\end{align*}
Applying integration by parts and Fubini's theorem, we have
\[
\begin{split}
&\int_{t_2}^{t_1}\int_{B_{\rho}(x_{0})}\frac{\partial u_{\varepsilon}}{\partial t}(t)\phi_{h}(t)\,\d \mu\,\d t -\left[\int_{B_{\rho}(x_0)}u_\varepsilon(t)\phi_{h}(t)\,\d \mu\right]_{t=t_2}^{t_1}
+\int_{t_2}^{t_1}\int_{B_{\rho}(x_{0})}(g_{u_{i}})_\varepsilon(t)\,\d\mu\,\d t \\
&\qquad=-\int_{t_2}^{t_1}\int_{B_{\rho}(x_{0})}u_{\varepsilon}(t)\frac{\partial\phi_{h}}{\partial t}(t)\,\d \mu\,\d t
+\int_{t_2}^{t_1}\int_{B_{\rho}(x_{0})}(g_{u_{i}})_\varepsilon(t)\,\d\mu\,\d t \\
&\qquad\le\int_{-\varepsilon}^{\varepsilon}\int_{t_2}^{t_1}\int_{B_{\rho}(x_{0})}g_{v_i+\phi}(t)\eta_{\varepsilon}(s)\,\d\mu\,\d t\,\d s
+\int_{t_2}^{t_1}\int_{B_{\rho}(x_{0})}g_{\phi_{h}-\phi}(t)\,\d\mu\,\d t\\
&\qquad\qquad-\left[\int_{B_{\rho}(x_0)}u_\varepsilon(t)\phi_{h}(t)\,\d \mu\right]_{t=t_2}^{t_1}
+\epsilon.
\end{split}
\]
Therefore,
\[
\begin{split}
&\int_{t_2}^{t_1}\int_{B_{\rho}(x_{0})}\frac{\partial u_{\varepsilon}}{\partial t}(t)\phi_{h}(t)\,\d \mu\,\d t 
+\int_{t_2}^{t_1}\int_{B_{\rho}(x_{0})}(g_{u_{i}})_\varepsilon(t)\,\d\mu\,\d t \\
&\qquad\le\int_{-\varepsilon}^{\varepsilon}\int_{t_2}^{t_1}\int_{B_{\rho}(x_{0})}g_{v_i+\phi}(t)\eta_{\varepsilon}(s)\,\d\mu\,\d t\,\d s
+\int_{t_2}^{t_1}\int_{B_{\rho}(x_{0})}g_{\phi_{h}-\phi}(t)\,\d\mu\,\d t
+\epsilon.
\end{split}
\]
Lemma \ref{vbthm} implies that the last term on the right-hand side converges to zero as $h\rightarrow 0$.
By passing to the limit $h\rightarrow 0$, we have
\begin{equation}\label{4dg}
\begin{split}
&\int_{t_2}^{t_1}\int_{B_{\rho}(x_{0})}\frac{\partial u_{\varepsilon}}{\partial t}(t)\phi(t)\,\d \mu\,\d t
+\int_{t_2}^{t_1}\int_{B_{\rho}(x_{0})}(g_{u_{i}})_\varepsilon(t)\,\d\mu\,\d t \\
&\qquad\leq\int_{-\varepsilon}^{\varepsilon}\int_{t_2}^{t_1}\int_{B_{\rho}(x_{0})}g_{v_i+\phi}(t)\eta_{\varepsilon}(s)\,\d\mu\,\d t\,\d s
+\epsilon.
\end{split}
\end{equation}

Let $\varphi\in\textrm{Lip}(\Omega_T)$ with $\supp\varphi\Subset B_{\rho}(x_{0})\times (0,T)$ and $0\leq\varphi\leq 1$. Let $\zeta_h$ be a cutoff function depending only on time defined as
\[
\zeta_h(t)
=\begin{cases}
\frac1h(t-t_0+\theta\rho+h),&\quad t_0-\theta\rho-h\le t<t_0-\theta\rho,\\
1,&\quad t_0-\theta\rho\le t\le t_0,\\
-\frac1h(t-t_0-h),&\quad t_0<t\le t_0+h,\\
0,&\quad\text{otherwise}.
\end{cases}
\]

We apply the test function $\phi=-\varphi\zeta_h((u_i)_{\varepsilon}-k)_{+}$ in \eqref{4dg}.
For the right-hand side of \eqref{4dg} we have
\begin{equation}\label{rhs}
\begin{split}
&\int_{-\varepsilon}^{\varepsilon}\int_{t_2}^{t_1}\int_{B_{\rho}(x_{0})}g_{v_i-\varphi\zeta_h((u_i)_{\varepsilon}-k)_{+}}(t)\eta_{\varepsilon}(s)\,\d\mu\,\d t\,\d s\\
&\qquad\leq\int_{-\varepsilon}^{\varepsilon}\int_{t_2}^{t_1}\int_{B_{\rho}(x_{0})}g_{v_i-u_i}(t)\eta_{\varepsilon}(s)\,\d\mu\,\d t\,\d s
+\int_{-\varepsilon}^{\varepsilon}\int_{t_2}^{t_1}\int_{B_{\rho}(x_{0})}g_{u_i-(u_i)_\varepsilon}(t)\eta_{\varepsilon}(s)\,\d\mu\,\d t\,\d s\\
&\qquad\qquad+\int_{-\varepsilon}^{\varepsilon}\int_{t_2}^{t_1}\int_{B_{\rho}(x_{0})}g_{(u_i)_\varepsilon-\varphi\zeta_h((u_i)_{\varepsilon}-k)_{+}}(t)\eta_{\varepsilon}(s)\,\d\mu\,\d t\,\d s\\
&\qquad= \int_{-\varepsilon}^{\varepsilon}\left(\int_{t_2}^{t_1}\int_{B_{\rho}(x_{0})}g_{v_i-u_i}(t)\,\d\mu\,\d t\right)\eta_{\varepsilon}(s)\,\d s
+\int_{t_2}^{t_1}\int_{B_{\rho}(x_{0})}g_{u_i-(u_i)_\varepsilon}(t)\,\d\mu\,\d t\\
&\qquad\qquad+\int_{t_2}^{t_1}\int_{B_{\rho}(x_{0})}g_{(u_i)_\varepsilon-\varphi\zeta_h((u_i)_{\varepsilon}-k)_{+}}(t)\,\d\mu\,\d t.
\end{split}
\end{equation}
Denote
\[
A_\varepsilon(t)
=\{x\in B_{\rho}(x_{0}):(u_i)_\varepsilon(x,t)>k\}.
\]
For the last integral in \eqref{rhs}, we have
\begin{align*}
&\int_{t_2}^{t_1}\int_{B_{\rho}(x_{0})}g_{(u_i)_\varepsilon-\varphi\zeta_h((u_i)_{\varepsilon}-k)_{+}}(t)\,\d\mu\,\d t\\
&\qquad=\int_{t_2}^{t_1}\int_{A_\varepsilon(t)}g_{(u_i)_\varepsilon-\varphi\zeta_h((u_i)_{\varepsilon}-k)}(t)\,\d\mu\,\d t
+\int_{t_2}^{t_1}\int_{B_{\rho}(x_{0})\setminus A_\varepsilon(t)}g_{(u_i)_\varepsilon}(t)\,\d\mu\,\d t.
\end{align*}
The Leibniz rule for upper gradients implies
\begin{align*}
&\int_{t_2}^{t_1}\int_{A_\varepsilon(t)}g_{(u_i)_\varepsilon-\varphi\zeta_h((u_i)_{\varepsilon}-k)}(t)\,\d\mu\,\d t
=\int_{t_2}^{t_1}\int_{A_\varepsilon(t)}g_{(1-\varphi\zeta_h)((u_i)_{\varepsilon}-k)}(t)\,\d\mu\,\d t\\
&\qquad\leq \int_{t_2}^{t_1}\int_{A_\varepsilon(t)}(1-\varphi(t)\zeta_h(t))g_{(u_i)_{\varepsilon}-k}(t)\,\d\mu\,\d t
+\int_{t_2}^{t_1}\int_{A_\varepsilon(t)}((u_i)_{\varepsilon}(t)-k)g_{\varphi\zeta_h}(t)\,\d \mu\,\d t\\
&\qquad=\int_{t_2}^{t_1}\int_{B_{\rho}(x_{0})}(1-\varphi(t)\zeta_h(t))g_{((u_i)_{\varepsilon}-k)_+}(t)\,\d\mu\,\d t
+\int_{t_2}^{t_1}\int_{B_{\rho}(x_{0})}((u_i)_{\varepsilon}(t)-k)_+g_{\varphi\zeta_h}(t)\,\d \mu\,\d t.
\end{align*}
For the integral on the right-hand side of \eqref{4dg}, we have
\[
\begin{split}
&\int_{-\varepsilon}^{\varepsilon}\int_{t_2}^{t_1}\int_{B_{\rho}(x_{0})}g_{v_i-\varphi\zeta_h((u_i)_{\varepsilon}-k)_{+}}(t)\eta_{\varepsilon}(s)\,\d\mu\,\d t\,\d s\\
&\le\int_{-\varepsilon}^{\varepsilon}\left(\int_{t_2}^{t_1}\int_{B_{\rho}(x_{0})}g_{v_i-u_i}(t)\,\d\mu\,\d t\right)\eta_{\varepsilon}(s)\,\d s
+\int_{t_2}^{t_1}\int_{B_{\rho}(x_{0})}g_{u_i-(u_i)_\varepsilon}(t)\,\d\mu\,\d t\\
&\qquad+\int_{t_2}^{t_1}\int_{B_{\rho}(x_{0})}(1-\varphi(t)\zeta_h(t))g_{((u_i)_{\varepsilon}-k)_+}(t)\,\d\mu\,\d t
+\int_{t_2}^{t_1}\int_{B_{\rho}(x_{0})}((u_i)_{\varepsilon}(t)-k)_+g_{\varphi\zeta_h}(t)\,\d \mu\,\d t\\
&\qquad+\int_{t_2}^{t_1}\int_{B_{\rho}(x_{0})\setminus A_\varepsilon(t)}g_{(u_i)_\varepsilon}(t)\,\d\mu\,\d t.
\end{split}
\]

By letting $h\to0$, we obtain
\[
\begin{split}
&\int_{-\varepsilon}^{\varepsilon}\int_{t_2}^{t_1}\int_{B_{\rho}(x_{0})}g_{v_i-\varphi\chi_{[t_0-\theta\rho,t_0]}((u_i)_{\varepsilon}-k)_{+}}(t)\eta_{\varepsilon}(s)\,\d\mu\,\d t\,\d s\\
&\qquad\le\int_{-\varepsilon}^{\varepsilon}\left(\int_{t_2}^{t_1}\int_{B_{\rho}(x_{0})}g_{v_i-u_i}(t)\,\d\mu\,\d t\right)\eta_{\varepsilon}(s)\,\d s
+\int_{t_2}^{t_1}\int_{B_{\rho}(x_{0})}g_{u_i-(u_i)_\varepsilon}(t)\,\d\mu\,\d t\\
&\qquad\qquad+\int_{t_2}^{t_1}\int_{B_{\rho}(x_{0})}(1-\varphi(t)\chi_{[t_0-\theta\rho,t_0]}(t))g_{((u_i)_{\varepsilon}-k)_+}(t)\,\d\mu\,\d t\\
&\qquad\qquad+\int_{t_2}^{t_1}\int_{B_{\rho}(x_{0})}\chi_{[t_0-\theta\rho,t_0]}(t)((u_i)_{\varepsilon}(t)-k)_+g_{\varphi}(t)\,\d \mu\,\d t
+\int_{t_2}^{t_1}\int_{B_{\rho}(x_{0})\setminus A_\varepsilon(t)}g_{(u_i)_\varepsilon}(t)\,\d\mu\,\d t.
\end{split}
\]
Lemma \ref{vbthm} implies that
\[
\lim_{\varepsilon\to0}\int_{-\varepsilon}^{\varepsilon}\left(\int_{t_2}^{t_1}\int_{B_{\rho}(x_{0})}g_{v_i-u_i}(t)\,\d\mu\,\d t\right)\eta_{\varepsilon}(s)\,\d s=0,
\]
and
\[
\lim_{\varepsilon\to0}\int_{t_2}^{t_1}\int_{B_{\rho}(x_{0})}g_{u_i-(u_i)_\varepsilon}(t)\,\d\mu\,\d t=0.
\]
By letting $\varepsilon\to0$, we have
\begin{align*}
\limsup_{\varepsilon\rightarrow0}
&\int_{-\varepsilon}^{\varepsilon}\int_{t_2}^{t_1}\int_{B_{\rho}(x_{0})}g_{v_i-\varphi\chi_{[t_0-\theta\rho,t_0]}((u_i)_{\varepsilon}-k)_{+}}(t)\eta_{\varepsilon}(s)\,\d\mu\,\d t\,\d s\\
&\leq\int_{t_2}^{t_1}\int_{B_{\rho}(x_{0})}(1-\varphi(t)\chi_{[t_0-\theta\rho,t_0]}(t))g_{(u_i-k)_+}(t)\,\d\mu\,\d t\\
&\qquad+\int_{t_2}^{t_1}\int_{B_{\rho}(x_{0})}\chi_{[t_0-\theta\rho,t_0]}(t)(u_i(t)-k)_+g_{\varphi}(t)\,\d \mu\,\d t
+\int_{t_2}^{t_1}\int_{B_{\rho}(x_{0})\setminus A(t)}g_{u_i}(t)\,\d\mu\,\d t,
\end{align*}
where 
\[
A(t)
=\{x\in B_{\rho}(x_{0}):u_i(x,t)>k\}.
\]

Here we used the following observation 
$$
\int_{-\infty}^{\infty}\vert\chi_{A_{\varepsilon}(t)}(x)-\chi_{A(t)}(x)\vert\,\d k
=\int_{\min\lbrace (u_i)_{\varepsilon}(x,t), u_i(x,t)\rbrace}^{\max\lbrace (u_i)_{\varepsilon}(x,t), u_i(x,t)\rbrace}1\,\d k
=\vert (u_i)_{\varepsilon}(x,t)-u_i(x,t)\vert.
$$
By Fubini's Theorem, we get
$$
\int_{-\infty}^{\infty}\int_{0}^{T}\int_{B_{\rho}(x_0)}\vert\chi_{A_{\varepsilon}(t)}-\chi_{A(t)}\vert\,\d \mu\,\d t\,\d k
=\int_{0}^{T}\int_{B_{\rho}(x_0)}\vert (u_i)_{\varepsilon}(t)-u_i(t)\vert\,\d \mu\,\d t.
$$
By Lemma \ref{vbthm}, $(u_i)_{\varepsilon}\rightarrow u_i$ in $L^1_{\textrm{loc}}(0,T,N^{1,1}_{\textrm{loc}}(\Omega))$, in particular we have $(u_i)_{\varepsilon}\rightarrow u_i$ in $L^1_{\textrm{loc}}(0,T,L^{1}_{\textrm{loc}}(\Omega))$ as $\varepsilon\rightarrow 0$. Therefore, 
$$
0=\lim_{\varepsilon\rightarrow 0}\int_{0}^{T}\int_{B_{\rho}(x_0)}\vert (u_i)_{\varepsilon}(t)-u_i(t)\vert\,\d \mu\,\d t
= \lim_{\varepsilon\rightarrow 0}\int_{-\infty}^{\infty}\int_{0}^{T}\int_{B_{\rho}(x_0)}\vert\chi_{A_{\varepsilon}(t)}-\chi_{A(t)}\vert\,\d \mu\,\d t\,\d k.
$$

By Fatou's lemma, we have
\begin{align*}
    0&\leq \int_{-\infty}^{\infty}\left(\lim_{\varepsilon\rightarrow 0}\int_{0}^{T}\int_{B_{\rho}(x_0)}\vert\chi_{A_{\varepsilon}(t)}-\chi_{A(t)}\vert\,\d \mu\,\d t\right)\,\d k\\
    &\leq \lim_{\varepsilon\rightarrow 0}\int_{-\infty}^{\infty}\int_{0}^{T}\int_{B_{\rho}(x_0)}\vert\chi_{A_{\varepsilon}(t)}-\chi_{A(t)}\vert\,\d \mu\,\d t\,\d k=0.
\end{align*}

This implies that
\begin{equation}\label{star}
\lim_{\varepsilon\rightarrow 0}\int_{0}^{T}\int_{B_{\rho}(x_0)}\vert\chi_{A_{\varepsilon}(t)}-\chi_{A(t)}\vert\,\d \mu\,\d t=0,
\end{equation}
for $\mathcal{L}^{1}$-almost every $k$.
Let $k\in\mathbb{R}$ be such that \eqref{star} holds. 
Applying Fatou's lemma one more time, we get
\[
 0\leq \int_{0}^{T}\left(\lim_{\varepsilon\rightarrow 0}\int_{B_{\rho}(x_0)}\vert \chi_{A_{\varepsilon}(t)}-\chi_{A(t)}\vert\,\d \mu\right)\,\d t
 \leq \lim_{\varepsilon\rightarrow 0}\int_{0}^{T}\int_{B_{\rho}(x_0)}\vert \chi_{A_{\varepsilon}(t)}-\chi_{A(t)}\vert\,\d \mu\,\d t=0.
\]

It follows that
\[
\lim_{\varepsilon\rightarrow 0}\int_{B_{\rho}(x_0)}\vert \chi_{A_{\varepsilon}(t)}-\chi_{A(t)}\vert\,\d \mu=0,
\]
for $\mathcal{L}^{1}$-almost every $t\in (0,T)$.
Therefore, we conclude $\chi_{A_{\varepsilon}(t)}\rightarrow \chi_{A(t)}$ in $L^{1}(B_{\rho}(x_{0}))$ as $\varepsilon\to0$, for $\mathcal{L}^{1}$-almost every $t\in (0,T)$ and for $\mathcal{L}^{1}$-almost every $k\in\mathbb{R}$.

For the first term on the left-hand side of \eqref{4dg} we find
\begin{align*}
\int_{t_2}^{t_1}\int_{B_{\rho}(x_{0})}\frac{\partial u_{\varepsilon}}{\partial t}(t)\phi(t)\,\d \mu\,\d t
&=-\int_{t_2}^{t_1}\int_{B_{\rho}(x_{0})}\frac{\partial u_{\varepsilon}}{\partial t}(t)(u_{\varepsilon}(t)-k)_{+}\varphi(t)\zeta_h(t)\,\d \mu\,\d t\\
&\xrightarrow{h\rightarrow 0}-\int_{t_2}^{t_1}\int_{B_{\rho}(x_{0})}\frac{\partial u_{\varepsilon}}{\partial t}(t)(u_{\varepsilon}(t)-k)_{+}\varphi(t)\chi_{[t_0-\theta\rho,t_0]}(t)\,\d \mu\,\d t,
\end{align*}
where
\begin{align*}
&-\int_{t_2}^{t_1}\int_{B_{\rho}(x_{0})}\frac{\partial u_{\varepsilon}}{\partial t}(t)(u_{\varepsilon}(t)-k)_{+}\varphi(t)\chi_{[t_0-\theta\rho,t_0]}(t)\,\d \mu\,\d t\\
&\qquad=-\int_{t_{2}}^{t_{1}}\int_{B_{\rho}(x_{0})}\frac{\partial u_{\varepsilon}}{\partial t}(t)(u_{\varepsilon}(t)-k)_{+}\varphi_{1}(t)\,\d \mu\,\d t\\
&\qquad=\frac{1}{2}\int_{t_2}^{t_1}\int_{B_{\rho}(x_{0})}\frac{\partial}{\partial t}\left((u_{\varepsilon}-k)_{+}^{2}\varphi_{1}\right)(t)\,\d \mu\,\d t
-\frac{1}{2}\int_{t_2}^{t_1}\int_{B_{\rho}(x_{0})}\frac{\partial\varphi_{1}}{\partial t}(t)(u_{\varepsilon}(t)-k)_{+}^{2}\,\d \mu\,\d t\\
&\qquad=\frac{1}{2}\left[\int_{B_{\rho}(x_{0})}(u_{\varepsilon}(t)-k)_{+}^{2}\varphi(t)\chi_{[t_0-\theta\rho,t_0]}(t)\,\d \mu\right]_{t=t_2}^{t_1}\\
&\qquad\qquad-\frac{1}{2}\int_{t_2}^{t_1}\int_{B_{\rho}(x_{0})}\frac{\partial\varphi}{\partial t}(t)(u_{\varepsilon}(t)-k)_{+}^{2}\chi_{[t_0-\theta\rho,t_0]}(t)\,\d \mu\,\d t\\
&\qquad\xrightarrow{\varepsilon\rightarrow 0}\frac{1}{2}\left[\int_{B_{\rho}(x_{0})}(u(t)-k)_{+}^{2}\varphi(t)\chi_{[t_0-\theta\rho,t_0]}(t)\,\d \mu\right]_{t=t_2}^{t_1}\\
&\qquad\qquad-\frac{1}{2}\int_{\max\lbrace t_{2},t_{0}-\theta\rho\rbrace}^{\min\lbrace t_1, t_{0}\rbrace}\int_{B_{\rho}(x_{0})}\frac{\partial\varphi}{\partial t}(t)(u(t)-k)_{+}^{2}\,\d \mu\,\d t.
\end{align*}

For the second term on the left-hand side of \eqref{4dg}, by Lemma \ref{vbthm}, we have
\[
\lim_{\varepsilon\rightarrow 0}\int_{t_2}^{t_1}\int_{B_{\rho}(x_{0})}(g_{u_{i}})_\varepsilon(t)\,\d\mu\,\d t 
=\int_{t_2}^{t_1}\int_{B_{\rho}(x_{0})}g_{u_i}(t)\,\d\mu\,\d t.
\]

Thus, by \eqref{4dg} we get
\begin{align*}
&\frac{1}{2}\left[\int_{B_{\rho}(x_{0})}(u(t)-k)_{+}^{2}\varphi(t)\chi_{[t_0-\theta\rho,t_0]}(t)\,\d \mu\right]_{t=t_{2}}^{t_{1}}
+\int_{t_2}^{t_1}\int_{B_{\rho}(x_{0})}g_{u_i}(t)\,\d\mu\,\d t\\
&\qquad\leq\frac{1}{2}\int_{\max\lbrace t_{2},t_{0}-\theta\rho\rbrace}^{\min\lbrace t_1, t_{0}\rbrace}\int_{B_{\rho}(x_{0})}\left\vert\frac{\partial\varphi}{\partial t}(t)\right\vert(u(t)-k)_{+}^{2}\,\d \mu\,\d t\\
&\qquad\qquad+\int_{t_2}^{t_1}\int_{B_{\rho}(x_{0})}(1-\varphi(t)\chi_{[t_0-\theta\rho,t_0]}(t))g_{(u_i-k)_+}(t)\,\d\mu\,\d t\\
&\qquad\qquad+\int_{t_2}^{t_1}\int_{B_{\rho}(x_{0})}(u_i(t)-k)_+g_{\varphi}(t)\chi_{[t_0-\theta\rho,t_0]}(t)\,\d \mu\,\d t
+\int_{t_2}^{t_1}\int_{B_{\rho}(x_{0})\setminus A(t)}g_{u_i}(t)\,\d\mu\,\d t,
\end{align*}
and consequently
\begin{align*}
&\frac{1}{2}\left[\int_{B_{\rho}(x_{0})}(u(t)-k)_{+}^{2}\varphi(t)\chi_{[t_0-\theta\rho,t_0]}(t)\,\d \mu\right]_{t=t_{2}}^{t_{1}}
+\int_{t_2}^{t_1}\int_{A(t)}g_{u_i}(t)\,\d\mu\,\d t\\
&\qquad\leq\frac{1}{2}\int_{\max\lbrace t_{2},t_{0}-\theta\rho\rbrace}^{\min\lbrace t_1, t_{0}\rbrace}\int_{B_{\rho}(x_{0})}\left\vert\frac{\partial\varphi}{\partial t}(t)\right\vert(u(t)-k)_{+}^{2}\,\d \mu\,\d t\\
&\qquad\qquad+\int_{t_2}^{t_1}\int_{B_{\rho}(x_{0})}(1-\varphi(t)\chi_{[t_2,t_1]}(t))g_{(u_i-k)_+}(t)\,\d\mu\,\d t\\
&\qquad\qquad+\int_{\max\lbrace t_{2},t_{0}-\theta\rho\rbrace}^{\min\lbrace t_1, t_{0}\rbrace}\int_{B_{\rho}(x_{0})}(u_i(t)-k)_+g_{\varphi}(t)\,\d \mu\,\d t.
\end{align*}
Since
\[
\int_{t_2}^{t_1}\int_{A(t)}g_{u_i}(t)\,\d\mu\,\d t
=\int_{t_2}^{t_1}\int_{B_{\rho}(x_{0})}g_{(u_i-k)_+}(t)\,\d\mu\,\d t,
\]
by absorbing terms, this implies
\[
\begin{split}
&\left[\frac{1}{2}\int_{B_{\rho}(x_{0})}(u(t)-k)_{+}^{2}\varphi(t)\chi_{[t_0-\theta\rho,t_0]}(t)\,\d \mu\right]_{t=t_{2}}^{t_{1}}+
\int_{t_2}^{t_1}\int_{B_{\rho}(x_{0})}\varphi(t)\chi_{[t_0-\theta\rho,t_0]}(t)g_{(u_i-k)_+}(t)\,\d\mu\,\d t\\
&\qquad\leq\frac{1}{2}\int_{\max\lbrace t_{2},t_{0}-\theta\rho\rbrace}^{\min\lbrace t_1, t_{0}\rbrace}\int_{B_{\rho}(x_{0})}\left\vert\frac{\partial\varphi}{\partial t}(t)\right\vert(u(t)-k)_{+}^{2}\,\d \mu\,\d t\\
&\qquad\qquad+\int_{\max\lbrace t_{2},t_{0}-\theta\rho\rbrace}^{\min\lbrace t_1, t_{0}\rbrace}\int_{B_{\rho}(x_{0})}g_{\varphi}(t)(u_i(t)-k)_+\,\d \mu\,\d t.
\end{split}
\]
By the Leibniz rule we obtain
\begin{align*}
&\int_{t_1}^{t_2}\int_{B_{\rho}(x_{0})}\varphi(t)\chi_{[t_0-\theta\rho,t_0]}(t)g_{(u_i-k)_+}(t)\,\d\mu\,\d t
=\int_{\max\lbrace t_{2},t_{0}-\theta\rho\rbrace}^{\min\lbrace t_1, t_{0}\rbrace}\int_{B_{\rho}(x_{0})}\varphi(t)g_{(u_i-k)_+}(t)\,\d\mu\,\d t\\
&\qquad\geq\int_{\max\lbrace t_{2},t_{0}-\theta\rho\rbrace}^{\min\lbrace t_1, t_{0}\rbrace}\int_{B_{\rho}(x_{0})}g_{\varphi(u_i-k)_+}(t)\,\d\mu\,\d t
-\int_{\max\lbrace t_{2},t_{0}-\theta\rho\rbrace}^{\min\lbrace t_1, t_{0}\rbrace}\int_{B_{\rho}(x_{0})}g_{\varphi}(t)(u_i(t)-k)_{+}\,\d \mu\,\d t.
\end{align*}
Thus we have
\[
\begin{split}
&\frac{1}{2}\left[\int_{B_{\rho}(x_{0})}(u(t)-k)_{+}^{2}\varphi(t)\chi_{[t_0-\theta\rho,t_0]}(t)\,\d \mu\right]_{t=t_{2}}^{t_{1}}
+\int_{\max\lbrace t_{2},t_{0}-\theta\rho\rbrace}^{\min\lbrace t_1, t_{0}\rbrace}\int_{B_{\rho}(x_{0})}g_{\varphi(u_i-k)_+}(t)\,\d\mu\,\d t\\
&\qquad\leq\frac{1}{2}\int_{\max\lbrace t_{2},t_{0}-\theta\rho\rbrace}^{\min\lbrace t_1, t_{0}\rbrace}\int_{B_{\rho}(x_{0})}\left\vert\frac{\partial\varphi}{\partial t}(t)\right\vert(u(t)-k)_{+}^{2}\,\d \mu\,\d t\\
&\qquad\qquad+2\int_{\max\lbrace t_{2},t_{0}-\theta\rho\rbrace}^{\min\lbrace t_1, t_{0}\rbrace}\int_{B_{\rho}(x_{0})}g_{\varphi}(t)(u_i(t)-k)_+\,\d \mu\,\d t.
\end{split}
\]

Letting $i\to\infty$, we obtain
\begin{equation}\label{genest}
\begin{split}
&\frac{1}{2}\left[\int_{B_{\rho}(x_{0})}(u(t)-k)_{+}^{2}\varphi(t)\chi_{[t_0-\theta\rho,t_0]}(t)\,\d \mu\right]_{t=t_{2}}^{t_{1}}
+\int_{\max\lbrace t_{2},t_{0}-\theta\rho\rbrace}^{\min\lbrace t_1, t_{0}\rbrace}\Vert D(\varphi(u-k)_{+})(t)\Vert(B_{\rho}(x_{0}))\,\d t\\
&\qquad\leq\frac{1}{2}\int_{\max\lbrace t_{2},t_{0}-\theta\rho\rbrace}^{\min\lbrace t_1, t_{0}\rbrace}\int_{B_{\rho}(x_{0})}\left\vert\frac{\partial\varphi}{\partial t}(t)\right\vert(u(t)-k)_{+}^{2}\,\d \mu\,\d t\\
&\qquad\qquad+2\int_{\max\lbrace t_{2},t_{0}-\theta\rho\rbrace}^{\min\lbrace t_1, t_{0}\rbrace}\int_{B_{\rho}(x_{0})}g_{\varphi}(t)(u(t)-k)_+\,\d \mu\,\d t.
\end{split}
\end{equation}
Choosing $t_{2}\in (0,T)$, with $t_{2}<t_0-\theta\rho$ then, for all $t_{0}-\theta\rho\leq t_{1}\leq t_{0}$, by \eqref{genest}, we have
\begin{equation*}
\begin{split}
&\frac{1}{2}\int_{B_{\rho}(x_{0})}(u(t_{1})-k)_{+}^{2}\varphi(t_1)\,\d \mu
\leq\frac{1}{2}\int_{B_{\rho}(x_{0})}(u(t_{1})-k)_{+}^{2}\varphi(t_1)\,\d \mu\\
&\qquad\qquad+\int_{t_0-\theta\rho}^{t_1}\Vert D(\varphi(u-k)_{+})(t)\Vert(B_{\rho}(x_{0}))\,\d t\\
&\qquad\leq\frac{1}{2}\int_{t_{0}-\theta\rho}^{t_1}\int_{B_{\rho}(x_{0})}\left\vert\frac{\partial\varphi}{\partial t}(t)\right\vert(u(t)-k)_{+}^{2}\,\d \mu\,\d t 
+2\int_{t_{0}-\theta\rho}^{t_1}\int_{B_{\rho}(x_{0})}g_{\varphi}(t)(u(t)-k)_+\,\d \mu\,\d t\\
&\qquad\qquad+\frac{1}{2}\int_{B_{\rho}(x_{0})}(u(t_{2})-k)_{+}^{2}\varphi(t_2)\chi_{[t_0-\theta\rho, t_0]}(t_2)\,\d \mu\\
&\qquad=\frac{1}{2}\int_{t_{0}-\theta\rho}^{t_1}\int_{B_{\rho}(x_{0})}\left\vert\frac{\partial\varphi}{\partial t}(t)\right\vert(u(t)-k)_{+}^{2}\,\d \mu\,\d t 
+2\int_{t_{0}-\theta\rho}^{t_1}\int_{B_{\rho}(x_{0})}g_{\varphi}(t)(u(t)-k)_+\,\d \mu\,\d t\\
&\qquad\leq \frac{1}{2}\int_{t_{0}-\theta\rho}^{t_0}\int_{B_{\rho}(x_{0})}\left\vert\frac{\partial\varphi}{\partial t}(t)\right\vert(u(t)-k)_{+}^{2}\,\d \mu\,\d t 
+2\int_{t_{0}-\theta\rho}^{t_0}\int_{B_{\rho}(x_{0})}g_{\varphi}(t)(u(t)-k)_+\,\d \mu\,\d t.
\end{split}
\end{equation*}
We conclude that
\begin{equation*}
\begin{split}
\frac{1}{2}\int_{B_{\rho}(x_{0})}(u(t_{1})-k)_{+}^{2}\varphi(t_{1})\,\d \mu
&\leq \frac{1}{2}\int_{t_{0}-\theta\rho}^{t_0}\int_{B_{\rho}(x_{0})}\left\vert\frac{\partial\varphi}{\partial t}(t)\right\vert(u(t)-k)_{+}^{2}\,\d \mu\,\d t \\
&\qquad+2\int_{t_{0}-\theta\rho}^{t_0}\int_{B_{\rho}(x_{0})}g_{\varphi}(t)(u(t)-k)_+\,\d \mu\,\d t,
\end{split}
\end{equation*}
and this holds for any $t_{0}-\theta\rho\leq t_{1}\leq t_{0}$.
Therefore
\begin{equation}\label{firstfinal}
\begin{split}
\esssup_{t_{0}-\theta\rho\leq t\leq t_{0}}\int_{B_{\rho}(x_{0})}(u(t)-k)_{+}^{2}\varphi(t)\,\d \mu&\leq \iint_{Q_{\rho,\theta}^{-}(x_{0},t_{0})}
\left\vert\frac{\partial\varphi}{\partial t}(t)\right\vert(u(t)-k)_{+}^{2}\,\d \mu\,\d t \\
&\qquad+4\iint_{Q_{\rho,\theta}^{-}(x_{0},t_{0})}g_{\varphi}(t)(u(t)-k)_+\,\d \mu\,\d t.
\end{split}
\end{equation}

On the other hand, by choosing $t_2=t_0-\theta\rho$ and $t_1=t_0$ in \eqref{genest}, we have
\begin{equation*}
\begin{split}
&\frac{1}{2}\int_{t_{0}-\theta\rho}^{t_0}\Vert D(\varphi(u-k)_{+})(t)\Vert(B_{\rho}(x_{0}))\,\d t\leq\int_{t_{0}-\theta\rho}^{t_0}\Vert D(\varphi(u-k)_{+})(t)\Vert(B_{\rho}(x_{0}))\,\d t\\
&\qquad\leq\frac{1}{2}\iint_{Q_{\rho,\theta}^{-}(x_{0},t_{0})}\left\vert\frac{\partial\varphi}{\partial t}(t)\right\vert(u(t)-k)_{+}^{2}\,\d \mu\,\d t
+2\iint_{Q_{\rho,\theta}^{-}(x_{0},t_{0})}g_{\varphi}(t)(u(t)-k)_+\,\d \mu\,\d t\\
&\qquad\qquad-\left[\frac{1}{2}\int_{B_{\rho}(x_{0})}(u(t)-k)_{+}^{2}\varphi(t)\,\d \mu\right]_{t=t_{0}-\theta\rho}^{t_{0}}.
\end{split}
\end{equation*}
This implies that
\begin{equation}\label{secondfinal}
\begin{split}
&\int_{t_{0}-\theta\rho}^{t_0}\Vert D(\varphi(u-k)_{+})(t)\Vert(B_{\rho}(x_{0}))\,\d t\leq\iint_{Q_{\rho,\theta}^{-}(x_{0},t_{0})}\left\vert\frac{\partial\varphi}{\partial t}(t)\right\vert(u(t)-k)_{+}^{2}\,\d \mu\,\d t\\
&\qquad+4\iint_{Q_{\rho,\theta}^{-}(x_{0},t_{0})}g_{\varphi}(t)(u(t)-k)_+\,\d \mu\,\d t-\left[\int_{B_{\rho}(x_{0})}(u(t)-k)_{+}^{2}\varphi(t)\,\d \mu\right]_{t=t_{0}-\theta\rho}^{t_{0}}.
\end{split}
\end{equation}

Adding \eqref{firstfinal} and \eqref{secondfinal} we conclude that
\begin{equation*}
\begin{split}
&\esssup_{t_{0}-\theta\rho\leq t\leq t_{0}}\int_{B_{\rho}(x_{0})}(u(t)-k)_{+}^{2}\varphi(t)\,\d \mu
+\int_{t_{0}-\theta\rho}^{t_{0}}\Vert D(\varphi(u-k)_+)(t)\Vert(B_{\rho}(x_{0}))\,\d t\\
&\qquad\leq 2\iint_{Q^{-}_{\rho, \theta}(x_{0},t_{0})}\left\vert\frac{\partial\varphi}{\partial t}(t)\right\vert(u(t)-k)_{+}^{2}\,\d \mu\,\d t
+8\iint_{Q^{-}_{\rho, \theta}(x_{0},t_{0})}(u(t)-k)_+g_{\varphi}(t)\,\d \mu\,\d t\\
&\qquad\qquad-\left[\int_{B_{\rho}(x_{0})}\varphi(t)(u(t)-k)_{+}^{2}\,\d \mu\right]_{t=t_{0}-\theta\rho}^{t_{0}}\\
&\qquad\leq 8\left(\iint_{Q^{-}_{\rho, \theta}(x_{0},t_{0})}\left\vert\frac{\partial\varphi}{\partial t}(t)\right\vert(u(t)-k)_{+}^{2}\,\d \mu\,\d t
+\iint_{Q^{-}_{\rho, \theta}(x_{0},t_{0})}(u(t)-k)_+g_{\varphi}(t)\,\d \mu\,\d t\right)\\
&\qquad\qquad-\left[\int_{B_{\rho}(x_{0})}\varphi(t)(u(t)-k)_{+}^{2}\,\d \mu\right]_{t=t_{0}-\theta\rho}^{t_{0}}.
\end{split}
\end{equation*}
This implies $u\in DG^{+}(\Omega_{T};8)$. A similar argument shows that $u\in DG^{-}(\Omega_{T};8)$ and thus $u\in DG(\Omega_{T};8)$

\end{proof}
\section{De Giorgi lemma}\label{DeGiorgiLemma}
This short section is devoted to prove that functions in a parabolic De Giorgi class are bounded from below. 
We apply the following standard iteration lemma in the proof, see \cite[Lemma 5.1]{dbgv}.
\begin{lemma}\label{lemma1}
Let $(Y_{n})_{n\in\N_0}$ be a sequence of positive numbers that satisfies 
\begin{equation}
Y_{n+1}\leq Cb^{n}Y_{n}^{1+\alpha},
\end{equation}
where $C, b>1$ and $\alpha>0$ are given numbers. If
$Y_{0}\leq C^{-1/\alpha}b^{-1/\alpha^{2}}$,
then $Y_{n}\rightarrow0$ as $n\rightarrow\infty$.
\end{lemma}

Let $\rho, \theta>0$ be such that $Q^{-}_{\rho, \theta}(x_{0},t_{0})\subset\Omega_{T}$ and let 
\[
m_{+}\geq\esssup_{Q^{-}_{\rho, \theta}(x_{0},t_{0})} u, 
\quad m_{-}\leq\essinf_{Q^{-}_{\rho, \theta}(x_{0},t_{0})} u
\quad\text{and}\quad  \omega\geq m_{+}-m_{-}.
\]

The following lemma is a version of \cite[Lemma 6.1]{dbgk} on metric measure spaces.

\begin{lemma}\label{degiorgi}
 Assume that  $u\in DG^{-}(\Omega_{T};\gamma)$.
\begin{itemize}
\item[{\rm (i)}] For $a, \xi\in (0,1)$ and $\overline{\theta}\in (0,\theta)$, there exists a constant $\nu_{-}=\nu_{-}(\gamma, C_{\mu}, C_{P}, \omega, \xi, a,\theta,\overline{\theta})$ such that if
\[
(\mu\otimes\mathcal{L}^{1})( Q^{-}_{\rho, \theta}(x_{0},t_{0})\cap\{u\leq m_{-}+\xi\omega\})
\leq\nu_{-}(\mu\otimes\mathcal{L}^{1})(Q^{-}_{\rho, \theta}(x_{0},t_{0})),
\]
then $u\geq m_{-}+a\xi\omega$ $(\mu\otimes\mathcal{L}^{1})$-almost everywhere in $B_{\frac{\rho}2}(x_{0})\times (t_{0}-\overline{\theta}\rho, t_{0}]$.
\item[{\rm (ii)}] For $a, \xi\in (0,1)$ and $\overline{\theta}\in (0,\theta)$, there exists a constant $\nu_{+}=\nu_{+}(\gamma, C_{\mu}, C_{P}, , \omega, \xi, a,\theta,\overline{\theta})$
such that if
\begin{equation*}
(\mu\otimes\mathcal{L}^{1})(Q^{-}_{\rho, \theta}(x_{0},t_{0})\cap\{u\geq m_{+}-\xi\omega\})
\leq\nu_{+}(\mu\otimes\mathcal{L}^{1})(Q^{-}_{\rho, \theta}(x_{0},t_{0})),
\end{equation*}
then $u\leq m_{+}-a\xi\omega$ $(\mu\otimes\mathcal{L}^{1})$-almost everywhere in $B_{\frac{\rho}2}(x_{0})\times (t_{0}-\overline{\theta}\rho, t_{0}]$.
\end{itemize}
\end{lemma}

\begin{proof}
We prove (i) and the proof for (ii) is similar.
For $n\in\N_0$, let
\[
\rho_{n}=\frac{\rho}{2}\left( 1+\frac{1}{2^{n}}\right), 
\quad\theta_{n}=\overline{\theta}+\frac{1}{2^{n}}(\theta-\overline{\theta})
\quad\text{and}\quad
t_{n}=t_{0}-\theta_{n}\rho.
\]
Then $\rho_{n}\searrow\frac{\rho}{2}$, 
$\theta_{n}\searrow\overline{\theta}$
and $t_{n}=\nearrow t_{0}-\theta\rho$ as $n\to\infty$.
Denote $B_{n}=B_{\rho_{n}}(x_{0})$ and $Q_{n}^{-}=B_{n}\times (t_{n},t_{0}]$.
Consider Lipschitz continuous functions $\zeta_{n}$, $n\in\N$, with
$\zeta_{n}=1$ in $Q_{n+1}^{-}$, $\zeta_{n}=0$ in $Q_{\rho,\theta}^{-}(x_{0},t_{0})\setminus Q_{n}^{-}$,
\begin{equation*}
g_{\zeta_{n}}\leq\frac{1}{\rho_{n}-\rho_{n+1}}=\frac{2^{n+2}}{\rho}
\quad\text{and}\quad 
0\leq(\zeta_{n})_{t}\leq\frac{2^{n+1}}{\theta-\overline{\theta}}\frac{1}{\rho}.
\end{equation*}
For $n\in\N$, let 
\[
\xi_{n}=a\xi+\frac{1-a}{2^{n}}\xi
\quad\text{and}\quad
k_{n}=m_{-}+\xi_{n}\omega.
\]
Then $\xi_{n}\searrow a\xi$ and $k_{n}\searrow m_{-}+a\xi\omega$ as $n\to\infty$.

Denote $A_{n}=Q_{n}^{-}\cap\{u\leq k_{n}\}$, $n\in\N$. By $\eqref{degiorgiclass}$ we have
\begin{align*}
&\esssup_{t_{n}< t<t_{0}}\int_{B_{n}}\zeta_{n}(t)(u(t)-k_{n})_{-}^{2}\,\d \mu
+\int_{t_{n}}^{t_{0}}\Vert D(\zeta_n(u-k_{n})_{-})(t)\Vert(B_{n})\,\d t \\
&\qquad\leq\gamma\iint_{Q_{n}}\left(g_{\zeta_n}(t)(u(t)-k_{n})_{-}
+\vert
(\zeta_{n})_{t}(t)
\vert
(u(t)-k_{n})_{-}^{2}\right)\,\d\mu\,\d t\\
&\qquad\leq\gamma\left(\frac{2^{n+2}}{\rho}\iint_{Q_{n}^{-}}(u(t)-k_{n})_{-}\,\d\mu\,\d t
+\frac{2^{n+1}}{(\theta-\overline{\theta})\rho}\iint_{Q_{n}^{-}}(u(t)-k_{n})_{-}^{2}\,\d\mu\,\d t\right).
\end{align*}
In $\{u\leq k_{n}\}$ we have
\begin{gather*}
0\leq k_{n}-u=m_{-}+\xi_{n}\omega-u=(m_{-}-u)+\xi_{n}\omega\leq\xi_{n}\omega\leq a\xi\omega,\\
0\leq (u-k_{n})_{-}\leq a\xi\omega
\quad\text{and}\quad
(u-k_{n})_{-}^{2}\leq a^{2}\xi^{2}\omega^{2}.
\end{gather*}
It follows that
\begin{align*}
&\frac{2^{n+2}}{\rho}\iint_{Q_{n}^{-}}(u(t)-k_{n})_{-}\,\d\mu\,\d t 
+\frac{2^{n+1}}{(\theta-\overline{\theta})\rho}\iint_{Q_{n}^{-}}(u(t)-k_{n})_{-}^{2}\,\d\mu\,\d t\\
&\qquad\leq\frac{2^{n+2}}{\rho}a\xi\omega(\mu\otimes\mathcal{L}^{1})(Q_{n}^{-}\cap\{u\leq k_{n}\}) 
+\frac{2^{n+1}}{(\theta-\overline{\theta})\rho}a^{2}\xi^{2}\omega^{2}(\mu\otimes\mathcal{L}^{1})(Q_{n}\cap\{u\leq k_{n}\})\\
&\qquad=(\mu\otimes\mathcal{L}^{1})(A_{n})\left(\frac{2^{n+2}}{\rho}a\xi\omega
+\frac{2^{n+1}}{(\theta-\overline{\theta})\rho}a^{2}\xi^{2}\omega^{2}\right).
\end{align*}
This implies
\begin{equation*}
\esssup_{t_{n}<t<t_{0}}\int_{B_{n}}\zeta_{n}(t)(u(t)-k_{n})_{-}^{2}\,\d\mu
+\int_{t_{n}}^{t_{0}}\Vert D(\zeta_{n}(u-k_{n})_{-})(t)\Vert(B_{n})\,\d t\leq 2^{n}\gamma_{1}\rho^{-1} (\mu\otimes\mathcal{L}^{1})(A_{n}),
\end{equation*}
where 
\[
\gamma_{1}=2\gamma a\xi\omega\frac{2(\theta-\overline{\theta})+a\xi\omega}{(\theta-\overline{\theta})}.
\]
By Proposition \ref{1ap} there exists a constant $C=C(C_\mu,C_P)$ such that, for $\kappa=\frac{Q+2}{Q}$, we have
\begin{equation}\label{V}
\begin{split}
&\iint_{Q_{n}^{-}}\left(\zeta_{n}(t)(u(t)-k_{n})_{-}\right)^{\kappa}\,\d\mu\,\d t\\
&\qquad\leq\frac{C\rho_{n}}{\mu(B_{n})^{\frac1Q}}
\int_{t_{n}}^{t_{0}}\Vert  D(\zeta_{n}(u-k_{n})_{-})(t)\Vert\,\d t
 \left(\esssup_{t_{n}<t<t_{0}}\int_{B_{n}}\zeta_{n}(t)(u(t)-k_{n})_{-}^{2}\,\d\mu\right)^{\frac{1}{Q}} \\
&\qquad\le\frac{C\rho_{n}}{\mu(B_{n})^{\frac1Q}}
\left(\esssup_{t_{n}<t<t_{0}}\int_{B_{n}} (u(t)-k_{n})_{-}^{2}\zeta_{n}(t)\,\d\mu
+\int_{t_{n}}^{t_{0}}\Vert  D((u-k_{n})_{-}\zeta_{n})(t)\Vert\,\d t\right)^{1+\frac1Q}\\
&\qquad\leq\frac{C\rho_{n}}{\mu(B_{n})^{\frac1Q}}
\left(2^{n}\gamma_{1}\rho^{-1} (\mu\otimes\mathcal{L}^{1})(A_{n})\right)^{1+\frac1Q}\\
&\qquad\le2^{n\frac{Q+1}{Q}}\rho^{-\frac1Q}\gamma_{2}\left(\frac{(\mu\otimes\mathcal{L}^{1})(A_{n})}{\mu(B_{n})}\right)^{1+\frac1Q}\mu(B_{n}),
\end{split}
\end{equation}
where $\gamma_{2}=C\gamma_{1}^{1+\frac1Q}$.
On the other hand, we have
\begin{equation}\label{VI}
\begin{split}
\iint_{Q_{n}^{-}}\left((u(t)-k_{n})_{-}\zeta_{n}(t)\right)^{\kappa}\,\d\mu\,\d t
&\geq \iint_{Q_{n+1}^{-}}\left((u(t)-k_{n})_{-}\zeta_{n}(t)\right)^{\kappa}\,\d\mu\,\d t\\
&\geq \iint_{A_{n+1}}(u(t)-k_{n})_{-}^{\kappa}\,\d\mu\,\d t\\
&=\iint_{A_{n+1}}(k_{n}-u(t))^{\kappa}\,\d\mu\,\d t\\
&\geq\iint_{A_{n+1}}(k_{n}-k_{n+1})^{\kappa}\,\d\mu\,\d t\\
&=2^{-\kappa n}\left(\frac{\omega\xi(1-a)}{2}\right)^{\kappa}(\mu\otimes\mathcal{L}^{1})(A_{n+1}).
\end{split}
\end{equation}

Let
\begin{equation*}
Y_{n}=\frac{(\mu\otimes\mathcal{L}^{1})(A_{n})}{(\mu\otimes\mathcal{L}^{1})(Q_{n}^{-})},
\end{equation*}
for $n\in\N$.
By \eqref{VI} and \eqref{V} we obtain
\begin{equation}\label{VII}
\begin{split}
Y_{n+1}
&=\frac{(\mu\otimes\mathcal{L}^{1})(A_{n+1})}{(\mu\otimes\mathcal{L}^{1})(Q_{n+1}^{-})} \\
&\leq\frac{1}{(\mu\otimes\mathcal{L}^{1})(Q_{n+1}^{-})}2^{\kappa n}
\left(\frac{\omega\xi(1-a)}{2}\right)^{-\kappa}
\iint_{Q_{n}^{-}}\left((u(t)-k_{n})_{-}\zeta_{n}(t)\right)^{\kappa}\,\d\mu\,\d t \\
&\leq\frac{\mu(B_{n})}{(\mu\otimes\mathcal{L}^{1})(Q_{n+1}^{-})}2^{\kappa n}\left(\frac{\omega\xi(1-a)}{2}\right)^{-\kappa}\left(2^{n(1+\frac1Q)}\rho^{-\frac1Q}\gamma_{2}
\left(\frac{(\mu\otimes\mathcal{L}^{1})(A_{n})}{\mu(B_{n})}\right)^{1+\frac1Q}\right) \\
&=\frac{\mu(B_{n})(\theta_{n}\rho)^{1+\frac1Q}}{\mu(B_{n+1})\theta_{n+1}\rho^{1+\frac1Q}}
2^{n(1+\frac1Q+\kappa)}\gamma_{2}
\left(\frac{\omega\xi(1-a)}{2}\right)^{-\kappa}Y_{n}^{1+\frac{1}{Q}} \\
&\leq\frac{\mu(B_{n})}{\mu(B_{n+1})}\frac{\theta_{n}}{\theta_{n+1}}b^{n}\gamma_{3}Y_{n}^{1+\frac{1}{Q}},
\end{split}
\end{equation}
where
\[
b=2^{1+\frac1Q+\kappa}
\quad\text{and}\quad
\gamma_{3}=\gamma_{2}\theta^{\frac1Q}\left(\frac{\omega\xi(1-a)}{2}\right)^{-\kappa}.
\]
By the doubling property we have
\begin{equation*}
\mu(B_{n})
=\mu(B_{\rho_{n}}(x_{0}))
\leq\mu(B_{2\rho_{n+1}}(x_{0}))
\leq C_{\mu}\mu(B_{\rho_{n+1}}(x_{0}))
=\mu(B_{n+1}),
\end{equation*}
and consequently
\[
\frac{\mu(B_{n})}{\mu(B_{n+1})}\frac{\theta_{n}}{\theta_{n+1}}\leq 2C_{\mu},
\]
for every $n\in\N$.
By \eqref{VII} we conclude
\begin{equation*}
Y_{n+1}\leq 2C_{\mu}b^{n}\gamma_{3}Y_{n}^{1+\frac{1}{Q}}=\gamma_{4}b^{n}Y_{n}^{1+\frac{1}{Q}},
\end{equation*}
where 
\begin{equation*}
\gamma_{4}=2C_\mu\gamma_{3}
=2^{\frac{Q-1}{Q}}C\left(\frac{\omega\xi}{\theta}\right)^{\frac1Q}(1-a)^{-\frac{Q+2}{Q}}
\left(a\gamma\frac{2(\theta-\overline{\theta})+a\xi\omega}{(\theta-\overline{\theta})}\right)^{\frac{Q+1}{Q}}.
\end{equation*}
By Lemma \ref{lemma1}, we have $Y_{n}\rightarrow 0$ as $n\rightarrow\infty$ provided
\begin{align*}
Y_{0}&\leq \gamma_{4}^{-Q}b^{-Q^{2}}
=2^{-(Q-1)}C\left(\frac{\omega\xi}{\theta}\right)^{-1}(1-a)^{Q+2}\left(a\gamma\frac{2(\theta-\overline{\theta})
+a\xi\omega}{(\theta-\overline{\theta})}\right)^{-(Q+1)}\\
&=\nu_{-}=\nu_{-}(\gamma, C_{\mu}, C_{P}, \omega, \xi, a,\theta,\overline{\theta}).
\end{align*}
The proof of (ii) is almost identical. One starts from inequalities \eqref{degiorgiclass} for the truncated functions $(u-k_{n})_{+}$ with $k_{n}=\mu_{+}-\xi_{n}\omega$ for the same choice of $\xi_{n}$.
\end{proof}

\section{Time expansion of positivity}\label{TimeExpansion}
In this section we prove an expansion of positivity result, which is a version of \cite[Lemma 7.1]{dbgk} on metric measure spaces. Roughly speaking, it asserts that information on the measure of the positivity set of $u$ at time level $t_0$ over the ball $B_{\rho}(x_0)$, translates into an expansion of positivity set in time (from $t_{0}$ to $t_0+\theta\rho$ for some suitable $\theta$). Most of the arguments and proofs are based on the energy estimates and De Giorgi Lemma of Section \ref{ParabolicDeGiorgiClass} and Section \ref{DeGiorgiLemma}.\

For a cylinder $Q_{2\rho,\theta}^{+}(x_{0},t_{0})=B_{2\rho}(x_{0})\times (t_{0},t_{0}+\theta\rho)\subset\Omega_{T}$, let
\begin{equation*}
m_{+}\geq\esssup_{Q_{2\rho,\theta}^{+}(x_{0},t_{0})} u,
\quad m_{-}\leq\essinf_{Q_{2\rho,\theta}^{+}(x_{0},t_{0})} u
\quad\text{and}\quad 
\omega\geq m_{+}-m_{-}.
\end{equation*}
The parameter $\theta$ will be determined by the proof.
Let $\xi\in (0,1)$ be a fixed parameter. 

\begin{lemma}\label{TimeExpansionPos}
Let $u\in DG^{-}(\Omega_{T};\gamma)$ and assume that
\begin{equation*}
\mu(\{x\in B_{\rho}(x_{0}):u(x,t_{0})\geq m_{-}+\xi\omega\}\geq\frac{1}{2}\mu(B_{\rho}(x_{0})),
\end{equation*}
for some $(x_{0},t_{0})\in\Omega_{T}$ and some $\rho>0$.
Then there exist $\delta=\delta(C_\mu, \gamma)\in(0,1)$ and $\varepsilon\in(0,1)$ such that
\begin{equation*}
\mu(\{x\in B_{\rho}(x_{0}):u(x,t)\geq m_{-}+\varepsilon\xi\omega\})
\geq\frac{1}{4}\mu(B_{\rho}(x_{0})),
\end{equation*}
for every $t\in(t_{0},t_{0}+\delta\xi\omega\rho)$.
\end{lemma}

\begin{proof}
Let $A_{k,\rho}(t)=\{x\in B_{\rho}(x_{0}):u(x, t)<k\}$ with $k>0$ and $t>0$.
Since
\begin{align*}
&\frac{1}{2}\mu(B_{\rho}(x_{0}))+\mu(\{x\in B_{\rho}(x_{0}):u(x,t_{0})<m_{-}+\xi\omega\})\\
&\qquad\leq\mu(\{x\in B_{\rho}(x_{0}):u(x,t_{0})\geq m_{-}+\xi\omega\})
+\mu(\{x\in B_{\rho}(x_{0}):u(x,t_{0})< m_{-}+\xi\omega\})\\
&\qquad\leq \mu(B_{\rho}(x_{0})),
\end{align*}
we have
\begin{equation*}
\mu(A_{m_{-}+\xi\omega,\rho}(t_{0}))
=\mu(\{x\in B_{\rho}(x_{0}):u(x,t_{0})< m_{-}+\xi\omega\})
\leq \frac{1}{2}\mu(B_{\rho}(x_{0})).
\end{equation*}
Let $\zeta$ be a Lipschitz cutoff function which is independent of $t$, $0\le\zeta\le1$, $\zeta=1$ in $B_{(1-\sigma)\rho}(x_{0})$ and $g_{\zeta}\leq\frac{1}{\sigma\rho}$,
where $\sigma\in(0,1)$ is to be chosen.
We apply the De Giorgi condition for $(u-(m_{-}+\xi\omega))_{-}$ in $Q_{\rho,\theta}^{+}(x_{0},t_{0})=B_{\rho}(x_{0})\times(t_{0},t_{0}+\theta\rho]$, where $\theta>0$ is to be chosen, and obtain
\begin{align*}
&\esssup_{t_{0}\leq t\leq t_{0}+\theta\rho}\int_{B_{\rho}(x_{0})}\zeta(t)(u(t)-(m_{-}+\xi\omega))_{-}^{2}\,\d\mu
+\int_{t_{0}}^{t_{0}+\theta\rho}\Vert D((u-(m_{-}+\xi\omega))_{-}\zeta)(t)\Vert(B_{\rho}(x_{0}))\,\d t\\
&\qquad\le\gamma\iint_{Q_{\rho,\theta}^{+}(x_0,t_0)}g_{\zeta}(t)(u(t)-(m_{-}+\xi\omega))_{-}\,\d\mu\,\d t
-\left[\int_{B_{\rho}(x_{0})}(u(t)-(m_{-}+\xi\omega))_{-}^{2}\zeta(t)\,\d\mu\right]_{t=t_{0}}^{t_0+\theta\rho}\\
&\qquad\le\gamma\iint_{Q_{\rho,\theta}^{+}(x_0,t_0)}g_{\zeta}(t)(u(t)-(m_{-}+\xi\omega))_{-}\,\d\mu\,\d t
+\int_{B_{\rho}(x_{0})}(u(t_0)-(m_{-}+\xi\omega))_{-}^{2}\zeta(t_0)\,\d\mu\\
&\qquad\le\frac{\gamma}{\sigma\rho}\iint_{Q_{\rho,\theta}^{+}(x_0,t_0)}(u(t)-(m_{-}+\xi\omega))_{-}\,\d\mu\,\d t
+\int_{B_{\rho}(x_{0})}(u(t_0)-(m_{-}+\xi\omega))_{-}^{2}\,\d\mu.
\end{align*}
Notice that for $t\in (t_{0},t_{0}+\theta\rho)$, we have
\begin{align*}
&\int_{B_{(1-\sigma)\rho}(x_{0})}(u(t)-(m_{-}+\xi\omega))_{-}^{2}\zeta(t)\,\d \mu
=\int_{B_{(1-\sigma)\rho}(x_{0})}(u(t)-(m_{-}+\xi\omega))_{-}^{2}\,\d\mu\\
&\qquad\leq\int_{B_{\rho}(x_{0})}(u(t)-(m_{-}+\xi\omega))_{-}^{2}\zeta(t)\,\d\mu\\
&\qquad\leq \esssup_{t_{0}\leq t\leq t_{0}+\theta\rho}\int_{B_{\rho}(x_{0})}\zeta(t)(u(t)-(m_{-}+\xi\omega))_{-}^{2}\,\d\mu\\
&\qquad\leq\frac{\gamma}{\sigma\rho}\iint_{Q_{\rho,\theta}^{+}(x_0,t_0)}(u(t)-(m_{-}+\xi\omega))_{-}\,\d\mu\,\d t
+\int_{B_{\rho}(x_{0})}(u(t_0)-(m_{-}+\xi\omega))_{-}^{2}\,\d\mu\\
&\qquad\leq\frac{\gamma}{\sigma\rho}\xi\omega\theta\rho\mu(B_{\rho}(x_{0}))
+\frac{(\xi\omega)^{2}}{2}\mu(B_{\rho}(x_{0}))\\
&\qquad=(\xi\omega)^{2}\left(\frac{\gamma\theta}{\sigma(\xi\omega)}+\frac{1}{2}\right)\mu(B_{\rho}(x_{0})).
\end{align*}
The last inequality holds, because
\begin{align*}
\int_{B_{\rho}(x_{0})}(u(t_{0})-(m_{-}+\xi\omega))_{-}^{2}\,\d\mu
&=\int_{A_{m_{-}+\xi\omega, \rho}(t_{0})}(m_{-}+\xi\omega-u(t_{0}))^{2}\,\d\mu\\
&\le(\xi\omega)^{2}\mu(A_{m_{-}+\xi\omega, \rho}(t_{0}))
\leq\frac{(\xi\omega)^{2}}{2}\mu(B_{\rho}(x_{0})),
\end{align*}
and
\begin{align*}
\iint_{Q_{\rho,\theta}^{+}(x_0,t_0)}(u(t)-(m_{-}+\xi\omega))_{-}\,\d\mu\,\d t
&=\int_{t_{0}}^{t_{0}+\theta\rho}\int_{A_{m_{-}+\xi\omega, \rho}(t_{0})}(m_{-}+\xi\omega-u(t))\,\d\mu\,\d t\\
&\leq\int_{t_{0}}^{t_{0}+\theta\rho}\xi\omega\int_{A_{m_{-}+\xi\omega, \rho}(t_{0})}\,\d\mu\,\d t\\
&\leq\xi\omega\mu(B_{\rho}(x_{0}))\theta\rho.
\end{align*}

Therefore
\begin{equation*}
\int_{B_{(1-\sigma)\rho}(x_{0})}(u(t)-(m_{-}+\xi\omega))_{-}^{2}\,\d\mu
\leq(\xi\omega)^{2}\left(\frac{\gamma\theta}{\sigma(\xi\omega)}+\frac{1}{2}\right)\mu(B_{\rho}(x_{0})),
\end{equation*}
for every $t\in (t_{0},t_{0}+\theta\rho)$.
The left-hand side can be estimated by
\begin{align*}
\int_{B_{(1-\sigma)\rho}(x_{0})}(u(t)-(m_{-}+\xi\omega))_{-}^{2}\,\d\mu
&\geq \int_{A_{m_{-}+\varepsilon\xi\omega, (1-\sigma)\rho}(t)}(u(t)-(m_{-}+\xi\omega))_{-}^{2}\,\d\mu\\
&\geq \int_{A_{m_{-}+\varepsilon\xi\omega, (1-\sigma)\rho}(t)}(m_{-}+\xi\omega-u(t))^{2}\,\d\mu\\
&>\int_{A_{m_{-}+\varepsilon\xi\omega, (1-\sigma)\rho}(t)}(m_{-}+\xi\omega-(m_{-}+\varepsilon\xi\omega))^{2}\,\d\mu\\
&=(\xi\omega)^{2}(1-\varepsilon)^{2}\mu(A_{m_{-}+\varepsilon\xi\omega, (1-\sigma)\rho}(t)),
\end{align*}
where $\varepsilon\in(0,1)$ is to be chosen.

By the doubling property and Bernoulli's inequality, we obtain
\begin{align*}
\mu(A_{m_{-}+\varepsilon\xi\omega, \rho}(t))&=\mu\left( A_{m_{-}+\varepsilon\xi\omega, (1-\sigma)\rho}(t)\cup\left(A_{m_{-}+\varepsilon\xi\omega, \rho}(t)\setminus A_{m_{-}+\varepsilon\xi\omega, (1-\sigma)\rho}(t)\right)\right)\\
&\leq\mu\left( A_{m_{-}+\varepsilon\xi\omega, (1-\sigma)\rho}(t)\right)+\mu(B_{\rho}(x_{0})\setminus B_{(1-\sigma)\rho}(x_{0}))\\
&\le\mu\left( A_{m_{-}+\varepsilon\xi\omega, (1-\sigma)\rho}(t)\right)
+\left(\mu(B_{\rho}(x_{0}))-\frac{\mu(B_{(1-\sigma)\rho}(x_{0}))}{\mu(B_{\rho}(x_{0}))}\mu(B_{\rho}(x_{0}))\right)\\
&\le\mu\left( A_{m_{-}+\varepsilon\xi\omega, (1-\sigma)\rho}(t)\right)+\mu(B_{\rho}(x_{0}))\left(1-C_{\mu}^{2}(1-\sigma)^{Q}\right)\\
&\leq\mu\left( A_{m_{-}+\varepsilon\xi\omega, (1-\sigma)\rho}(t)\right)+\mu(B_{\rho}(x_{0}))\left(1-(1-\sigma)^{Q}\right)\\
&\leq\mu\left( A_{m_{-}+\varepsilon\xi\omega, (1-\sigma)\rho}(t)\right)
+Q\sigma\mu(B_{\rho}(x_{0})).
\end{align*}
Combining these estimates gives
\begin{align*}
\mu(A_{m_{-}+\varepsilon\xi\omega, \rho}(t))
&\leq\frac{1}{(\xi\omega)^{2}(1-\varepsilon)^{2}}\int_{B_{(1-\sigma)\rho}(x_{0})}(u(t)-(m_{-}+\xi\omega))_{-}^{2}\,\d\mu
+Q\sigma\mu(B_{\rho}(x_{0}))\\
&\leq\frac{1}{(\xi\omega)^{2}(1-\varepsilon)^{2}}
\left((\xi\omega)^{2}\left(\frac{\gamma\theta}{\sigma\xi\omega}+\frac{1}{2}\right)\mu(B_{\rho}(x_{0}))\right)
+Q\sigma\mu(B_{\rho}(x_{0}))\\
&=\mu(B_{\rho}(x_{0}))\left(\frac{1}{(1-\varepsilon)^{2}}\left(\frac{\gamma\theta}{\sigma\xi\omega}+\frac{1}{2}\right)+Q\sigma\right)\\
&\leq\frac{\mu(B_{\rho}(x_{0}))}{(1-\varepsilon)^{2}}\left(\frac{\gamma\theta}{\sigma\xi\omega}+\frac{1}{2}+Q\sigma\right).
\end{align*}
Setting $\theta=\frac{\xi\omega}{2^{8}\gamma Q}$ and $\sigma=\frac{1}{16Q}$, we obtain
\[
\frac{1}{(1-\varepsilon)^{2}}\left(\frac{\gamma\theta}{\sigma\xi\omega}+\frac{1}{2}+Q\sigma\right)
=\frac{1}{(1-\varepsilon)^{2}}\left(\frac{\frac{1}{2^{8}Q}}{\frac{1}{16Q}}+\frac{1}{2}+\frac{1}{16}\right)
=\frac{1}{(1-\varepsilon)^{2}}\frac{5}{8}
<\frac{3}{4}.
\]
By letting $0<\varepsilon\leq\frac{1}{32}$, we have
\begin{equation*}
\mu(\{x\in B_{\rho}(x_{0}):u(x,t)>m_{-}+\varepsilon\xi\omega\})+\mu(A_{m_{-}+\varepsilon\xi\omega, \rho}(t))\geq\mu(B_{\rho}(x_{0})),
\end{equation*}
and thus
\begin{align*}
\mu(\{x\in B_{\rho}(x_{0}):u(x,t)>m_{-}+\varepsilon\xi\omega\})
&\geq\mu(B_{\rho}(x_{0}))-\mu(A_{m_{-}+\varepsilon\xi\omega, \rho}(t))\\
&\geq\mu(B_{\rho}(x_{0}))-\frac{3}{4}\mu(B_{\rho}(x_{0}))
=\frac{1}{4}\mu(B_{\rho}(x_{0})).
\end{align*}
Therefore, the claim holds for $\delta=\frac{1}{2^8\gamma Q}$.
\end{proof}

\section{Characterization of continuity}\label{CharacterizationContinuity}
Finally we are ready to prove the main result of this paper. 
\begin{theorem}\label{NecSuf}
Let $u\in L_{\loc}^{1}(0,T;BV_{\loc}(\Omega))$ be a variational solution to the total variation flow in $\Omega_{T}$. 
Then $u$ is continuous at some $(x_{0},t_{0})\in \Omega_{T}$ if and only if
\begin{equation*}
\lim_{\rho\to 0+}\frac{\rho}{(\mu\otimes\mathcal{L}^{1})(Q_{\rho,1}^{-}(x_{0},t_{0}))}\int_{t_{0}-\rho}^{t_{0}}\Vert Du(t)\Vert(B_{\rho}(x_{0}))\,\d t=0.
\end{equation*}
\end{theorem}

\begin{proof}
We begin with the necessary part of Theorem \ref{NecSuf}.
By Proposition \ref{belonging}, we have $u\in DG(\Omega_{T};\gamma)$ with $\gamma=8$.
Assume that  $u$ is continuous at $(x_{0},t_{0})\in\Omega_{T}$. 
Without loss of generality we may assume $u(x_{0},t_{0})=0$. 
Let $\zeta$ be a Lipschitz cutoff function with $0\le\zeta\le1$,
$\zeta=0$ on $(X\times\mathbb R)\setminus Q_{2\rho,1}^{-}(x_{0},t_{0})$,
$\zeta=1$ on $Q_{\frac{3}{2}\rho,1}^{-}(x_{0},t_{0})$,
$\zeta(\cdot, t_{0}-2\rho)=0$,
$\zeta_t\ge0$
and $g_{\zeta}+\zeta_{t}\leq\frac{3}{\rho}$.
We apply \eqref{degiorgiclass} with $\theta=1$, $k=0$ and neglect the supremum term of the left-hand side to obtain
\[
\int_{t_{0}-2\rho}^{t_{0}}\Vert D(u_{+}\zeta)(t)\Vert(B_{2\rho}(x_{0}))\,\d t
\leq\gamma\iint_{Q_{2\rho,1}^{-}(x_{0},t_{0})}(u_{+}(t)g_{\zeta}(t)+u_{+}(t)^{2}\vert\zeta_{t}(t)\vert)\,\d \mu\,\d t,
\]
and
\[
\int_{t_{0}-2\rho}^{t_{0}}\Vert D(u_{-}\zeta)(t)\Vert(B_{2\rho}(x_{0}))\,\d t
\leq\gamma\iint_{Q_{2\rho,1}^{-}(x_{0},t_{0})}(u_{-}(t)g_{\zeta}(t)+u_{-}(t)^{2}\vert\zeta_{t}(t)\vert)\,\d \mu\,\d t.
\]
By adding up the inequalities above and using the doubling property of the measure, we obtain
\[
\begin{split}
\int_{t_{0}-2\rho}^{t_{0}}\Vert D(u\zeta)(t)\Vert(B_{2\rho}(x_{0}))\,\d t
&\leq\gamma\iint_{Q_{2\rho,1}^{-}(x_{0},t_{0})}(g_{\zeta}(t)\vert u(t)\vert+\vert\zeta_{t}(t)\vert u(t)^{2})\,\d \mu\,\d t\\
&\leq\frac{3\gamma}{\rho}\iint_{Q_{2\rho,1}^{-}(x_{0},t_{0})}(\vert u(t)\vert+u(t)^{2})\,\d \mu\,\d t\\
&=\frac{3\gamma}{\rho}(\mu\otimes\mathcal{L}^{1})(Q_{2\rho,1}^{-}(x_{0},t_{0})) \fiint_{Q_{2\rho,1}^{-}(x_{0},t_{0})}(\vert u(t)\vert+u(t)^{2})\,\d \mu\,\d t\\
&\leq\frac{6C_{\mu}\gamma}{\rho}(\mu\otimes\mathcal{L}^{1})(Q_{\rho,1}^{-}(x_{0},t_{0})) \fiint_{Q_{2\rho,1}^{-}(x_{0},t_{0})}(\vert u(t)\vert+u(t)^{2})\,\d \mu\,\d t.
\end{split}
\]
Since $u\zeta=u$ in $Q_{\frac{3}{2}\rho,1}^{-}(x_{0},t_{0})\supseteq Q_{\rho,1}^{-}(x_{0},t_{0})$, we obtain
\[
\begin{split}
&\frac{\rho}{(\mu\otimes\mathcal{L}^{1})(Q_{\rho,1}^{-}(x_{0},t_{0}))}\int_{t_{0}-2\rho}^{t_{0}}\Vert Du(t)\Vert(B_{\rho}(x_{0}))\,\d t\\
&\qquad\le\frac{\rho}{(\mu\otimes\mathcal{L}^{1})(Q_{\rho,1}^{-}(x_{0},t_{0}))}\int_{t_{0}-2\rho}^{t_{0}}\Vert D(u\zeta)(t)\Vert(B_{2\rho}(x_{0}))\,\d t\\
&\qquad\leq 6C_{\mu}\gamma\fiint_{Q_{2\rho,1}^{-}(x_{0},t_{0})}(\vert u(t)\vert+u(t)^{2})\,\d \mu\,\d t.
\end{split}
\]
The right-hand side tends to zero as $\rho\rightarrow 0$ implying the necessary condition of Theorem \ref{NecSuf}.

Let us then prove the sufficient part of Theorem \ref{NecSuf}.
Let $(x_{0},t_{0})\in\Omega_{T}$ and let $\rho>0$ be so small that $Q_{\rho, 1}^{-}(x_{0},t_{0})=B_{\rho}(x_{0})\times (t_{0}-\rho, t_{0}]\subset \Omega_{T}$. 
Set
\[
m_{+}=\esssup_{Q_{\rho,1}^{-}(x_{0},t_{0})} u,
\quad m_{-}=\essinf_{Q_{\rho,1}^{-}(x_{0},t_{0})} u
\quad\text{and}\quad 
\omega= m_{+}-m_{-}=\essosc_{Q_{\rho,1}^{-}(x_{0},t_{0})}u.
\]
Without loss of generality we may assume that $\omega\leq 1$ so that
\begin{equation*}
Q_{\rho,\omega}^{-}(x_{0},t_{0})=B_{\rho}(x_{0})\times(t_{0}-\omega\rho,t_{0}]\subset Q_{\rho,1}^{-}(x_{0},t_{0})\subset\Omega_{T}.
\end{equation*}
Therefore, 
\[
\essinf_{Q_{\rho,\omega}^{-}(x_{0},t_{0})} u\geq m_{-},
\quad\esssup_{Q_{\rho,\omega}^{-}(x_{0},t_{0})} u\leq m_{+}
\quad\text{and}\quad 
\omega\geq\essosc_{Q_{\rho,\omega}^{-}(x_{0},t_{0})} u.
\]
For a contradiction, assume that $u$ is not continuous at $(x_{0},t_{0})$.
Then there exists $\rho_{0}>0$ and $\omega_{0}>0$ such that
\begin{equation*}
\omega_{\tilde{\rho}}=\essosc_{Q_{\tilde{\rho},1}^{-}(x_{0},t_{0})} u\geq \omega_{0}>0,
\end{equation*}
for all $0<\tilde{\rho}\leq\rho_{0}$.
Let $\tilde{\delta}=\frac{1}{2^{8}\gamma Q}$, determined as in the proof of Lemma \ref{TimeExpansionPos} at the time level $\tilde{t}=t_{0}-\frac{\tilde{\delta}\omega\rho}{2}$.
Clearly
\begin{equation*}
\mu\left(\left\{x\in B_{\rho}(x_{0}):u\big(x,t_{0}-\tfrac{\tilde{\delta}\omega\rho}{2}\big)\geq m_{-}+\frac{\omega}{2}\right\}\right)
\geq\frac{1}{2}\mu(B_{\rho}(x_{0})),
\end{equation*}
or
\begin{equation*}
\mu\left(\left\{x\in B_{\rho}(x_{0}):u\big(x,t_{0}-\tfrac{\tilde{\delta}\omega\rho}{2}\big)\leq m_{+}+\frac{\omega}{2}\right\}\right)
\geq\frac{1}{2}\mu(B_{\rho}(x_{0})).
\end{equation*}
Assuming the former holds, by Lemma \ref{TimeExpansionPos} there is a $\delta$, actually $\delta=\tilde{\delta}$ works, and $\varepsilon=\frac{1}{32}$ such that
\begin{equation*}
\mu\left(\left\{x\in B_{\rho}(x_{0}):u(x,t)\geq m_{-}+\frac{\omega}{64}\right\}\right)
\geq\frac{1}{4}\mu(B_{\rho}(x_{0})),
\end{equation*}
for every $t\in (t_{0}-\frac{\tilde{\delta}\omega\rho}{2}, t_{0}]$.
Let $2\tilde{\xi}=\frac{1}{64}\tilde{\delta}$. Since $\frac{1}{64}\omega\geq\frac{1}{64}\tilde{\delta}\omega$, then 
\[
\left\{x\in B_{\rho}(x_{0}):u(x,t)\geq m_{-}+\frac{\omega}{64}\right\}
\subset\left\{x\in B_{\rho}(x_{0}):u(x,t)\geq m_{-}+\frac{\tilde{\delta}\omega}{64}\right\},
\] 
and thus
\[
\mu(\{x\in B_{\rho}(x_{0}):u(x,t)>m_{-}+2\tilde{\xi}\omega\})
\geq\mu\left(\left\{x\in B_{\rho}(x_{0}):u(x,t)\geq m_{-}+\frac{\omega}{64}\right\}\right)
\geq\frac{1}{4}\mu (B_{\rho}(x_{0})),
\]
for every $t\in (t_{0}-\frac{\tilde{\delta}\omega\rho}{2}, t_{0}]$.
Since $(t_{0}-\tilde{\xi}\omega\rho, t_{0}]\subset (t_{0}-\frac{\tilde{\delta}\omega\rho}{2}, t_{0}]$, we have
\begin{equation}\label{fourthBall}
\mu(\{x\in B_{\rho}(x_{0}):u(x,t)>m_{-}+2\tilde{\xi}\omega\})
\geq\frac{1}{4}\mu (B_{\rho}(x_{0})),
\end{equation}
for every $ t\in (t_{0}-\tilde{\xi}\omega\rho, t_{0}]$.
Next, we apply Lemma \ref{DeGiorgiBV} to the function $u(\cdot, t)$, for $t$ in the range $(t_{0}-\tilde{\xi}\omega\rho, t_{0}]$ over the ball $B_{\rho}(x_{0})$ 
with $k=m_{-}+\tilde{\xi}\omega$ and $l=m_{-}+2\tilde{\xi}\omega$, so that $l-k=\tilde{\xi}\omega$.

By the doubling property of the measure and \eqref{fourthBall}, we have
\begin{align*}
\frac{\tilde{\xi}\omega}{2^{Q+2}C_{\mu}}
&\leq\frac{\tilde{\xi}\omega}{4}\frac{\mu(B_{\rho}(x_{0}))}{\mu(B_{2\rho}(x_{0}))}
\leq\frac{\tilde{\xi}\omega
\mu(\{x\in B_{\rho}(x_{0}):u(x,t)>m_{-}+2\tilde{\xi}\omega\})}
{\mu(B_{2\rho}(x_{0}))}\\
&\leq C\rho\frac{\Vert Du(t)\Vert(\{x\in B_{\rho}(x_{0}):u(x,t)>m_{-}+\tilde{\xi}\omega\})}{\mu(\{x\in B_{\rho}(x_{0}):u(x,t)<m_{-}+\tilde{\xi}\omega\})}.
\end{align*}
This implies
\begin{equation*}
\tilde{\xi}\omega\mu(\{x\in B_{\rho}(x_{0}):u(x,t)<m_{-}+\tilde{\xi}\omega\})
\leq C\rho\Vert Du(t)\Vert(\{x\in B_{\rho}(x_{0}):u(x,t)>m_{-}+\tilde{\xi}\omega\}),
\end{equation*}
where $C=C(C_\mu, C_P)$.

Integrating over the time interval $(t_{0}-\tilde{\xi}\omega\rho, t_{0}]$ gives
\begin{align*}
\tilde{\xi}\omega
\int_{t_{0}-\tilde{\xi}\omega\rho}^{t_{0}}&\mu(\{x\in B_{\rho}(x_{0}):u(x,t)<m_{-}+\tilde{\xi}\omega\})\,\d t\\
&\leq C\rho \int_{t_{0}-\tilde{\xi}\omega\rho}^{t_{0}}\Vert Du(t)\Vert(\{x\in B_{\rho}(x_{0}):u(x,t)>m_{-}+\tilde{\xi}\omega\})\,\d t\\
&\leq C\rho \int_{t_{0}-\tilde{\xi}\omega\rho}^{t_{0}}\Vert Du(t)\Vert(B_{\rho}(x_{0}))\,\d t\\
&=C\rho\frac{(\mu\otimes\mathcal{L}^{1})(Q^{-}_{\rho,\tilde{\xi}\omega}(x_{0},t_{0}))}{(\mu\otimes\mathcal{L}^{1})(Q^{-}_{\rho,\tilde{\xi}\omega}(x_{0},t_{0}))}
\int_{t_{0}-\tilde{\xi}\omega\rho}^{t_{0}}\Vert Du(t)\Vert(B_{\rho}(x_{0}))\,\d t\\
&=\frac{C\rho}{\tilde{\xi}\omega}
\frac{(\mu\otimes\mathcal{L}^{1})(Q^{-}_{\rho,\tilde{\xi}\omega}(x_{0},t_{0}))}{(\mu\otimes\mathcal{L}^{1})(Q^{-}_{\rho,1}(x_{0},t_{0}))}
\int_{t_{0}-\tilde{\xi}\omega\rho}^{t_{0}}\Vert Du(t)\Vert(B_{\rho}(x_{0}))\,\d t.
\end{align*}
Since
\begin{equation*}
\int_{t_{0}-\tilde{\xi}\omega\rho}^{t_{0}}\mu(\{x\in B_{\rho}(x_{0}):u(x,t)<m_{-}+\tilde{\xi}\omega\})\,\d t
\geq (\mu\otimes\mathcal{L}^{1})(Q^{-}_{\rho,\tilde{\xi}\omega}(x_{0},t_{0})\cap\{u<m_{-}+\tilde{\xi}\omega\}),
\end{equation*}
we have
\begin{equation*}
\frac{(\mu\otimes\mathcal{L}^{1})(Q^{-}_{\rho,\tilde{\xi}\omega}(x_{0},t_{0})\cap\{u<m_{-}+\tilde{\xi}\omega\})}{(\mu\otimes\mathcal{L}^{1})(Q^{-}_{\rho,\tilde{\xi}\omega}(x_{0},t_{0}))}
\leq\frac{C}{(\tilde{\xi}\omega_{0})^{2}}\frac{\rho}{(\mu\otimes\mathcal{L}^{1})(Q^{-}_{\rho,1}(x_{0},t_{0}))}\int_{t_{0}-\tilde{\xi}\omega\rho}^{t_{0}}\Vert Du(\cdot, t)\Vert(B_{\rho}(x_{0}))\,\d t.
\end{equation*}
By assumption, the right-hand side tends to zero as $\rho\to 0+$. Hence, there exits $\rho>0$ small enough such that
\begin{equation*}
\frac{(\mu\otimes\mathcal{L}^{1})(Q^{-}_{\rho,\tilde{\xi}\omega}(x_{0},t_{0})\cap\{u<m_{-}+\tilde{\xi}\omega\})}{(\mu\otimes\mathcal{L}^{1})(Q^{-}_{\rho,\tilde{\xi}\omega}(x_{0},t_{0}))}\leq \nu_{-},
\end{equation*}
where $\nu_{-}$ is the number in Lemma \ref{degiorgi} for such a choice of parameters.
Lemma \ref{degiorgi} implies $u\geq m_{-}+\frac{1}{2}\tilde{\xi}\omega$ $(\mu\otimes\mathcal{L}^{1})$-almost everywhere in $Q^{-}_{\frac{1}{2}\rho,\tilde{\xi}\omega}(x_{0},t_{0})$ and consequently
\begin{equation*}
\essinf_{Q^{-}_{\frac{1}{2}\rho,\tilde{\xi}\omega}(x_{0},t_{0})} u\geq m_{-}+\frac{\tilde{\xi}\omega}{2}.
\end{equation*}
This implies
\begin{align*}
\essosc_{Q^{-}_{\frac{1}{2}\rho,\tilde{\xi}\omega}(x_{0},t_{0})} u
&=\esssup_{Q^{-}_{\frac{1}{2}\rho,\tilde{\xi}\omega}(x_{0},t_{0})} u -\essinf_{Q^{-}_{\frac{1}{2}\rho,\tilde{\xi}\omega}(x_{0},t_{0})}u
\leq\esssup_{Q^{-}_{\rho,1}(x_{0},t_{0})}u-m_{-}-\frac{\tilde{\xi}\omega}{2}\\
&=m_{+}-m_{-}-\frac{\tilde{\xi}\omega}{2}
=\omega-\frac{\tilde{\xi}\omega}{2}
=\left(1-\frac{\tilde{\xi}}{2}\right)\omega
=\eta\omega,
\end{align*}
for some $\eta\in (0,1)$.
With $\rho_{1}=\frac{1}{2}\tilde{\xi}\omega\rho$ we have $Q^{-}_{\rho_{1},1}(x_{0},t_{0})\subset Q^{-}_{\frac{1}{2}\rho,\tilde{\xi}\omega}(x_{0},t_{0})$ and thus
\begin{equation*}
\esssup_{Q^{-}_{\rho_{1},1}(x_{0},t_{0})} u\leq\esssup_{Q^{-}_{\frac{1}{2}\rho,\tilde{\xi}\omega}(x_{0},t_{0})}u
\quad\text{and}\quad 
\essinf_{Q^{-}_{\frac{1}{2}\rho,\tilde{\xi}\omega}(x_{0},t_{0})}u\leq\esssup_{Q^{-}_{\rho_{1},1}(x_{0},t_{0})}u.
\end{equation*}
Therefore, we have
\begin{align*}
\omega_{\rho_{1}}&=\essosc_{Q^{-}_{\rho_{1},1}(x_{0},t_{0})}u=\esssup_{Q^{-}_{\rho_{1},1}(x_{0},t_{0})}u-\essinf_{Q^{-}_{\rho_{1},1}(x_{0},t_{0})}u\\
&\leq \esssup_{Q^{-}_{\frac{1}{2}\rho,\tilde{\xi}\omega}(x_{0},t_{0})}u-\essinf_{Q^{-}_{\frac{1}{2}\rho,\tilde{\xi}\omega}(x_{0},t_{0})}u
=\essosc_{Q^{-}_{\frac{1}{2}\rho,\tilde{\xi}\omega}(x_{0},t_{0})}\leq\eta\omega.
\end{align*}
By repeating the same argument starting from the cylinder $Q^{-}_{\rho_{1},1}(x_{0},t_{0})$ and proceeding recursively, we generate a decreasing sequence of radii $\rho_{n}\rightarrow 0$ such that
\begin{equation*}
\omega_{0}\leq \essosc_{Q^{-}_{\rho_{n},1}(x_{0},t_{0})}u\leq\eta^{n}\omega,
\end{equation*}
for every $n\in\mathbb{N}$.
This is a contradiction with the assumption $u$ is not continuous at $(x_{0},t_{0})$.
\end{proof}

\renewcommand{\refname}{References}


\begin{thebibliography}{99}

\bibitem{AmbrosioDiMarino}\textsc{L. Ambrosio, S. Di Marino}, \textit{Equivalent definitions of BV space and of total variation on metric measure spaces},  J. Funct. Anal., \textbf{266} (2014), 4150--4188.

\bibitem{Anzellotti:1984}
\textsc{G. Anzellotti}, 
\textit{Pairings between measures and bounded functions and compensated compactness}, Ann. Mat. Pura Appl., \textbf{135} (1984), 293--318.

\bibitem{AndreuBallesterCasellesMazon}\textsc{F. Andreu, C. Ballester, V. Caselles, J.M. Maz\'on}, 
\textit{Minimizing total variation flow}, Differ. Integral Equ., \textbf{14} (2001), 321--360.

\bibitem{AndreuCasellesDiazMazon}\textsc{F. Andreu, V. Caselles, J.I. Diaz, J.M. Maz\'on}, 
\textit{Some qualitative properties for the total variation flow},  
J. Funct. Anal., \textbf{188} (2002), 516--547.

\bibitem{AndreuCasellesMazon}\textsc{F. Andreu-Vaillo, V. Caselles, J.M. Maz\'on}, 
\textit{Parabolic quasilinear equations minimizing linear growth functional},  
Prog. Math., 223, Birkh\"auser Verlag, Basel, 2004.

\bibitem{BellettiniCasellesNovaga}\textsc{G. Bellettini, V. Caselles, M. Novaga}, 
\textit{The total variation flow in $\mathbb{R}^N$},  
J. Differ. Equ., \textbf{184} (2002), 475--525.

\bibitem{bjorn} \textsc{A. Bj\"orn, J. Bj\"orn}, 
\textit{Nonlinear potential theory on metric spaces},  
EMS Tracts in Mathematics, European Mathematical Society, Z\"urich, 2011.

\bibitem{Buffa}
\textsc{V. Buffa}, 
\textit{Time-smoothing for parabolic variational problems in metric measure spaces}, 
Ann. Univ. Ferrara Sez. VII Sci. Mat. (to appear), \url{https://arxiv.org/abs/2002.00093}.

\bibitem{BuffaCollinsPacchiano}
\textsc{V. Buffa, M. Collins, C. Pacchiano}, 
\textit{Existence of parabolic minimizers to the total variation flow on metric measure spaces}, Manuscr. Math. (2022), \url{https://doi.org/10.1007/s00229-021-01350-2}.

\bibitem{BogeleinDuzaarMarcellini}
\textsc{V. B\"ogelein, F. Duzaar, P. Marcellini}, 
\textit{A time dependent variational approach to image restoration}, SIAM J. Imaging Sci., \textbf{8} (2015), 968--1006.

\bibitem{BogeleinDuzaarScheven}
\textsc{V. B\"ogelein, F. Duzaar, C. Scheven}, 
\textit{The total variation flow with time dependent boundary values}, Calc. Var. Partial Differential Equations, \textbf{55} (2016), no. 4, Art. 108. 

 \bibitem{BoegelDuzSchev:2015}
  \textsc{V. B\"ogelein, F. Duzaar, C. Scheven},
  \textit{The obstacle problem for the total variation flow}, Ann. Sci. \'Ec. Norm. Sup\'er. (4) \textbf{49} (2016), 1143--1188.

\bibitem{collins}
\textsc{M. Collins, A. Her\'an}, 
\textit{Existence of parabolic minimizers on metric measure spaces}, 
Nonlinear Anal., \textbf{176} (2018), 56--83.

\bibitem{cozzi}
\textsc{M. Cozzi},
\textit{Regularity results and Harnack inequalities for minimizers and solutions of nonlocal problems: a unified approach via fractional De Giorgi classes},
J. Funct. Anal., \textbf{272} (2017), 4762--4837. 

\bibitem{dbgk} \textsc{E. DiBenedetto, U. Gianazza, C. Klaus}, 
\textit{A necessary and sufficient condition for the continuity of local minima of parabolic variational integrals with linear growth},  
Adv. Calc. Var., \textbf{10} (2017), 209--221.

\bibitem{dbgv} \textsc{E. DiBenedetto, U. Gianazza, V. Vespri}, 
\textit{Harnack's inequality for degenerate and singular parabolic equations},  
Springer Monographs in Mathematics, Springer-Verlag, New York, 2012.

\bibitem{EvansGariepy}
\textsc{L.C. Evans, R.F. Gariepy},
\textit{Measure theory and fine properties of functions}, 
Studies in Advanced Mathematics, CRC Press, Boca Raton, 1992. 

\bibitem{FujishimaHabermann}
\textsc{Y. Fujishima, J. Habermann},
\textit{The stability problem for parabolic quasiminimizers in metric measure spaces},
 Atti Accad. Naz. Lincei Rend. Lincei Mat. Appl., \textbf{29} (2018), 343--376.

\bibitem{Fujishima_et_al}
\textsc{Y. Fujishima, J. Habermann, M. Masson},  
\textit{A fairly strong stability result for parabolic quasiminimizers}, 
Math. Nachr., \textbf{291} (2018), 1269--1282.

\bibitem{FujishimaHabermannKinnunenMasson}
\textsc{Y. Fujishima, J. Habermann, J. Kinnunen, M. Masson},
\textit{Stability for parabolic quasiminimizers}, 
Potential Anal., \textbf{41} (2014), 983--1004.

\bibitem{Gorny-Mazon:2021}
\textsc{W. G\'orny, J.M. Maz\'on},
\textit{The Neumann and Dirichlet problems for the total variation flow in metric measure spaces},
Preprint 2021.  \url{https://arxiv.org/abs/2004.09243}.

\bibitem{GianazzaKlaus}
\textsc{U. Gianazza, C. Klaus},
\textit{$p$-parabolic approximation of total variation flow solutions},
 Indiana Univ. Math. J., \textbf{68} (2019), 1519--1550. 

\bibitem{hei} \textsc{J. Heinonen}, 
\textit{Lectures on analysis on metric spaces},  Universitext,
Springer-Verlag, New York, 2001.

\bibitem{Heinonen1998}
\textsc{J. Heinonen and P. Koskela},
\textit{Quasiconformal maps in metric spaces with controlled geometry},
Acta Math., \textbf{181} (1998), 1--61.


\bibitem{hkst} \textsc{J. Heinonen, P. Koskela, N. Shanmugalingam, J. Tyson}, 
\textit{Sobolev spaces on metric measure spaces. An approach based on upper gradients},  
New Mathematical Monographs, 27, Cambridge University Press, Cambridge, 2015.

\bibitem{Heran}
\textsc{A. Her\'an}, 
\textit{Harnack inequality for parabolic quasi minimizers on metric spaces},
 Atti Accad. Naz. Lincei Rend. Lincei Mat. Appl., \textbf{31} (2020), 565--592.

\bibitem{Ivert_et_al}
\textsc{P.-A. Ivert, N. Marola, M. Masson},
\textit{Energy estimates for variational minimizers of a parabolic doubly nonlinear equation on metric measure spaces},
Ann. Acad. Sci. Fenn. Math., \textbf{39} (2014), 711--719. 

\bibitem{KinnunenMasson}
\textsc{J. Kinnunen, M. Masson},  
\textit{Parabolic comparison principle and quasiminimizers in metric measure spaces},
Proc. Amer. Math. Soc., \textbf{143} (2015), 621--632.

\bibitem{KinnunenShanmugalingam}
\textsc{J. Kinnunen, N. Shanmugalingam},
\textit{Regularity of quasi-minimizers on metric spaces},
Manuscripta Math., \textbf{105} (2001), 401--424.

\bibitem{KinnunenEtAl2014}
\textsc{J. Kinnunen, R. Korte, N. Shanmugalingam, H. Tuominen}, 
\textit{Pointwise properties of functions of bounded variation in metric spaces}, 
Rev. Mat. Complut., \textbf{27} (2014), 41--67.

\bibitem{kmmp} \textsc{J. Kinnunen, N. Marola, M. Miranda Jr., F. Paronetto}, 
\textit{Harnack's inequality for parabolic De Giorgi classes in metric spaces},  
Adv. Differ. Equ., \textbf{17} (2012), 801--832.

\bibitem{Koskela1998}
\textsc{P. Koskela, P. MacManus},
\textit{Quasiconformal mappings and Sobolev spaces},
Studia Math., \textbf{131} (1998), 1--17.

\bibitem{Kuttler} \textsc{K. Kuttler}, \textit{Modern analysis}, 
Studies in Advanced Mathematics, 26, CRC Press, 1998.

\bibitem{Lahti2017}
\textsc{P. Lahti}, \textit{Quasiopen sets, bounded variation and lower semicontinuity in metric spaces}, Potential Anal., \textbf{52} (2020), 321--337.

\bibitem{MarolaMasson}
\textsc{N. Marola, M. Masson},
\textit{On the Harnack inequality for parabolic minimizers in metric measure spaces},
Tohoku Math. J., \textbf{65} (2013), 569--589.

\bibitem{Masson_et_al_2013}
\textsc{M. Masson, M. Miranda Jr., F. Paronetto, M. Parviainen}, 
\textit{Local higher integrability for parabolic quasiminimizers in metric spaces},
Ric. Mat., \textbf{62} (2013), 279--305. 

\bibitem{MassonParviainen}
\textsc{M. Masson, M. Parviainen}, 
\textit{Global higher integrability for parabolic quasiminimizers in metric measure spaces}, 
J. Anal. Math., \textbf{126} (2015), 307--339.

\bibitem{MassonSiljander}
\textsc{M. Masson, J. Siljander},
\textit{H\"older regularity for parabolic De Giorgi classes in metric measure spaces}, 
Manuscr. Math., \textbf{142} (2013), 187--214.

\bibitem{mir}
\textsc{M. Miranda Jr},
\textit{Functions of bounded variation on ``good'' metric spaces},
J. Math. Pures Appl., \textbf{82} (2003), 975--1004.

\bibitem{ruz}
\textsc{M. R\r{u}\v{z}i\v{c}ka}, 
\textit{Nichtlineare Funktionalanalysis: Eine Einf\"{u}hrung}, Springer Lehrbuch Masterclass, Springer Berlin Heidelberg, 2004.


\bibitem{Shanmugalingam}
\textsc{N. Shanmugalingam},
\textit{Newtonian spaces: an extension of Sobolev spaces to metric measure spaces}, Rev. Mat. Iberoamericana, \textbf{16} (2000), 243--279.


\end{thebibliography}
\end{document}